\documentclass{IEEEtran}

\usepackage{cite}
\usepackage{amsmath,amssymb,amsfonts}
\usepackage[capitalise]{cleveref}
\usepackage[]{footmisc}
\usepackage{xcolor}
\usepackage{algorithm}
\usepackage{algorithmic}
\usepackage[algo2e]{algorithm2e} 
\usepackage{graphicx}
\usepackage{textcomp}
\usepackage{psfrag}
\usepackage{epsfig}
\usepackage{subcaption}
\usepackage{float}
\usepackage{xcolor}
\usepackage{booktabs}
\usepackage{multicol}
\usepackage{arydshln}
\usepackage{multirow}
\usepackage{wrapfig}
\usepackage[numbers,sort&compress]{natbib}

\newtheorem{mydef}{Definition}
\newtheorem{mypro}{Proposition}
\newtheorem{myass}{Assumption}
\newtheorem{myrem}{Remark}
\newtheorem{mylem}{Lemma}
\newtheorem{mythr}{Theorem}

\newtheorem{mycor}{Corollary}

\newtheorem{myproblem}{Problem}

\SetKwRepeat{Repeat}{Repeat}{Until}
\SetKw{Return}{Return}

\begin{document}
\title{\LARGE {\color{blue}Data-Driven Co-Design of Event-Triggered and Sparse Control for Resource-Aware Networked Control Systems}}

\author{Zhaohua~Yang,~Xiaoxu~Lyu,~Dawei~Shi,~\IEEEmembership{Senior Member,~IEEE},~and~Ling~Shi,~\IEEEmembership{Fellow,~IEEE}
	\thanks{\rm Zhaohua Yang and Xiaoxu Lyu are with the Department of Electronic and Computer Engineering, The Hong Kong University of Science and
		Technology, Clear Water Bay, Hong Kong, China (email: zyangcr@connect.ust.hk; eelyuxiaoxu@ust.hk).}
	\thanks{Dawei Shi is with the State Key Laboratory of
 Intelligent Control and Decision of Complex Systems and MIIT Key
 Laboratory of Servo Motion System Drive and Control, School of Automation, Beijing Institute of Technology, Beijing 100081, China (e-mail: daweishi@bit.edu.cn).}
\thanks{Ling Shi is with the Department of Electronic and Computer Engineering, the Department of Chemical and Biological Engineering, and The Cheng Kar-Shun Robotics Institute (CKSRI), The Hong Kong University of Science and Technology, Hong Kong, China (e-mail: eesling@ust.hk).}
\thanks{\rm The work is supported by the Hong Kong RGC General Research Fund 16203723.}
  }
\maketitle

\begin{abstract}
This paper investigates the data-driven co-design of event-triggered control (ETC) and sparse control (SC) for networked control systems (NCSs) with unknown linear dynamics. {\color{blue}While ETC and SC have been widely studied as effective strategies to reduce communication and computation burdens on different resource dimensions, existing works typically address them separately and rely on accurate system models. Furthermore, their joint design in a data-driven setting, especially in the presence of measurement and process noise, remains largely unexplored. To bridge these gaps, we propose a unified data-driven framework that simultaneously accounts for bounded state and input measurement noise as well as process noise, and enables the co-design of ETC mechanisms and sparse controllers directly from data. Within this framework, we characterize stability, uniformly ultimately bounded (UUB) behavior, and $H_\infty$ performance under different noise conditions.} For each problem, given the event-triggered parameters, we provide a sufficient condition for the existence of a feasible controller and develop an iterative algorithm to solve the associated nonconvex optimization problem. Numerical examples are provided to demonstrate the effectiveness of the proposed methods.
\end{abstract}


\IEEEpeerreviewmaketitle
\section{Introduction}
{\color{blue}Networked control systems (NCSs) have found widespread applications in diverse fields, and have attracted considerable research attention over the past decades. In such systems, information is exchanged among sensors, controllers, and actuators \cite{chen2025self}. However, practical implementations are often constrained by limited communication and computation resources, which restrict the scalability and applicability of these systems.} The operational cost mainly arises from communication over the sensor-to-controller (S2C) and controller-to-actuator (C2A) channels, as well as from intra-controller computation. These considerations motivate the adoption of event-triggered control (ETC) and sparse control (SC) as effective strategies for reducing resource consumption. Specifically, ETC reduces communication load by transmitting data only when certain event-triggered conditions are violated, while SC reduces dataflow burden and computational complexity by promoting sparsity in the controller gain matrix. Both approaches provide principled ways to balance control performance and resource usage~\cite{jovanovic2016controller,heemels2012introduction}.

{\color{blue}The ETC mechanism can be broadly categorized into static and dynamic ETC. In static ETC, the triggering condition is a static function of the available signals, without introducing additional dynamic variables \cite{heemels2012introduction,tabuada2007event,heemels2012periodic}, whereas dynamic ETC incorporates auxiliary dynamics to regulate the triggering behavior \cite{girard2014dynamic}. In parallel, SC research can be mainly divided into two directions: sparsity-promoting control and structured control. The sparsity-promoting approach investigates the trade-off between system performance and controller sparsity \cite{lin2013design,babazadeh2016sparsity,masazade2012sparsity}. 
The structured control direction focuses on synthesizing controllers with a prescribed sparsity pattern \cite{lin2011augmented,fardad2014design,ferrante2019design}, often dictated by communication or implementation constraints. }

The aforementioned literature on ETC and SC generally assumes perfect knowledge of the system dynamics. However, in practical applications, the system dynamics are often inaccessible or difficult to identify, motivating the development of data-driven control approaches.

Data-driven control has attracted significant attention in recent years as a promising strategy for situations where the system dynamics are unknown \textit{a priori} and only a sequence of measured data is available. In practice, measurements are often corrupted by noise, which motivates robust data-driven methods that explicitly account for uncertainty. Existing approaches can be broadly categorized into indirect and direct methods. Indirect methods first identify a model from data and then design controllers \cite{dean2020sample,ferizbegovic2019learning}. In contrast, direct methods bypass system identification and synthesize controllers directly from data, often leveraging Willems et al.'s fundamental lemma \cite{willems2005note}. Building on this foundation, recent works have addressed fundamental control problems such as stabilization and optimal control under noisy data \cite{van2020data,de2019formulas,de2021low}. More recently, robust data-driven approaches have been developed to explicitly handle bounded noise by constructing uncertainty sets of system matrices consistent with the data. For instance, Bisoffi et al. \cite{bisoffi2024controller} considered robust stabilization under measurement noise, while van Waarde et al. \cite{van2020noisy,van2023quadratic} introduced the matrix S-lemma to address $H_2$ and $H_\infty$ control problems. Alternative formulations based on matrix ellipsoids were proposed in \cite{bisoffi2022data}. 

Within the data-driven control framework, several works have investigated ETC and SC. For ETC, Digge and Pasumarthy \cite{digge2022data} analyzed discrete-time ETC using noise-free data. Feng et al. \cite{feng2025data} proposed a co-design approach for event-triggered parameters and controller gains with event-triggered measurements. Persis et al. \cite{de2023event} studied ETC under both noise-free and noisy data scenarios, providing key results on minimal inter-event times. Regarding SC, Eising and Cort{\'e}s \cite{eising2022informativity} developed a robust stabilizer with maximal sparsity, while Miller et al. \cite{miller2025data} designed structured controllers to improve $H_2$ performance. {\color{blue}Despite these advances, existing studies have not addressed the joint co-design of ETC and SC in a data-driven setting. More importantly, such a combination is nontrivial due to several fundamental challenges. First, constructing a general data-driven framework that simultaneously accounts for measurement noise and process noise requires careful reorganization. Second, event-triggered mechanisms introduce state-dependent uncertainties into the closed-loop system, since control updates are performed based on outdated state and input information. Third, sparse control design itself involves bilinear coupling between decision variables, which further complicates the synthesis problem. In particular, when both objectives are considered simultaneously, the resulting co-design problem becomes intrinsically more challenging, as sparsity-promoting objectives and event-triggered parameters cannot be jointly incorporated into a unified convex formulation. To the best of our knowledge, existing data-driven ETC and SC approaches have not simultaneously addressed these challenges. Nevertheless, addressing these challenges is essential for enabling efficient resource utilization in practical control systems. From a practical perspective, reducing communication and computation burdens is crucial in many real-world applications. For instance, in Unmanned Aerial Vehicle (UAV) networks \cite{gupta2015survey}, lowering the frequency of control updates reduces communication load and energy consumption, while sparse control decreases onboard computational demand. Similarly, in Industrial Internet of Things (IIoT) systems \cite{wu2019performance}, reducing transmission frequency alleviates network congestion, whereas lowering per-update computational cost enables scalable deployment across resource-constrained devices. These two aspects are inherently coupled, and optimizing only one may lead to inefficient or even infeasible implementations. 

Motivated by these challenges, we propose a unified data-driven co-design framework of ETC and SC for systems with unknown dynamics. The main contributions are summarized as follows.}
\begin{enumerate}
    \item We propose a novel data-driven framework that simultaneously accounts for bounded state measurement noise, input measurement noise, and process noise. Within this framework, we address the co-design of ETC and SC for unknown systems. Specifically, our approach enables the design of sparse controllers given the event-triggered parameters, which reflect the desired communication load in the S2C and C2A channels.
    \item {\color{blue}We adopt a unified cost modeling perspective that captures the interplay between communication and computation induced by event-triggering and controller sparsity. This perspective reveals an inherent trade-off that necessitates the co-design of ETC and sparse control for unknown systems. Accordingly, our method enables the co-design of ETC and SC to meet various control objectives under different assumptions on process noise. Specifically, in the absence of process noise, we design the sparsest stabilizing controller for given event-triggered parameters. With bounded process noise, we characterize the trade-off between the size of the uniformly ultimately bounded (UUB) set and controller sparsity, given event-triggered parameters. For square-summable process noise, we investigate the trade-off between the $H_\infty$ performance and controller sparsity, also under given event-triggered parameters.}
    \item For each problem, we provide a sufficient condition for the existence of a feasible controller and develop an iterative algorithm to solve the associated nonconvex optimization problem. The effectiveness of our methods is demonstrated through several numerical examples.
\end{enumerate}
The remainder of this paper is organized as follows. Section~\ref{Preliminaries} introduces the event-triggered control, sparse control, and unified cost model in NCSs, and formulates the problems of interest. Section~\ref{Main Results} presents the main results, including the data collection mechanism and construction of the uncertainty set, stability analysis without process noise, UUB analysis with process noise, and $H_\infty$ performance analysis with process noise. Section~\ref{Simulations} presents numerical examples to demonstrate the effectiveness, scalability, and advantages of the proposed methods. Finally, Section~\ref{Conclusions} concludes the paper.

\textit{Notations}: Let $\mathbb{R}$, $\mathbb{Z}$, and $\mathbb{N}$ denote the sets of real numbers, integers, and nonnegative integers, respectively. Let $\mathbb{Z}_{[i,j)}$ denote the set of integers in the interval $[i, j)$. Let $\mathbb{R}^{m \times n}$ and $\mathbb{R}^n$ denote the sets of real matrices of size $m \times n$ and real column vectors of size $n$, respectively. For a matrix $X$, let $X^{-1}$, $X^\top$, and $\mathrm{Tr}(X)$ denote its inverse, transpose, and trace, respectively. Let $X \succeq 0$ indicate that $X$ is positive semidefinite. Let $X_{ij}$ denote the $(i,j)$-th entry of $X$, $X^{|\cdot|}$ the element-wise absolute value, and $X^{\circ -1}$ the element-wise inverse. The space $l_2[0, +\infty)$ denotes the set of square-summable sequences. The notation $\triangleq$ denotes equality by definition, and $\forall$ means ``for all''. The operator $\mathrm{blkdiag}(\cdot)$ denotes a block diagonal matrix. {\color{blue}The $l_0$ norm $\|X\|_0$ denotes the number of nonzero entries of $X$.
The $l_1$ norm $\|X\|_1$ denotes the sum of absolute values of all entries of $X$.
The Frobenius norm $\|X\|_F$ is defined as $\|X\|_F \triangleq \sqrt{\mathrm{Tr}(X^\top X)}$}. The sets $\mathbb{S}^n$, $\mathbb{S}^n_{+}$, and $\mathbb{S}^n_{++}$ denote the sets of symmetric, symmetric positive semidefinite, and symmetric positive definite matrices of size $n \times n$, respectively. The notations $I$, $0$, and $\mathbf{1}$ denote identity, zero, and all-ones matrices with appropriate dimensions. {\color{blue}We sometimes abbreviate $\left[\begin{smallmatrix}X_1 & X_2 \\ X_2^\top & X_3\end{smallmatrix}\right]$ as  $\begin{bmatrix}\begin{smallmatrix}X_1 & X_2 \\ * & X_3\end{smallmatrix}\end{bmatrix}$ to save space. For compactness in LMI derivations, the following block-matrix constructions are occasionally used as shorthand for repeatedly appearing structures: $\mathcal{B}_1(X_1,X_2,X_3) \triangleq \begin{bmatrix}\begin{smallmatrix}0 &0 &0\\X_1 &0 &0\\X_2 &-X_2 &X_3 \end{smallmatrix}\end{bmatrix}$, $\mathcal{B}_2(X_1,X_2) \triangleq \begin{bmatrix}\begin{smallmatrix}X_1 &X_2\\-X_1 &0\\ 0&0 \end{smallmatrix}\end{bmatrix}$, $\mathcal{B}_3(X) \triangleq \begin{bmatrix}\begin{smallmatrix}X \\-X \\0 \end{smallmatrix}\end{bmatrix}$, and $\mathcal{B}_4(X) \triangleq \begin{bmatrix}\begin{smallmatrix}X &0 &0\\0 &0 &0\\0 &0 &0 \end{smallmatrix}\end{bmatrix}$.}
\begin{figure}[t]
    \centering    \includegraphics[width=1\columnwidth]{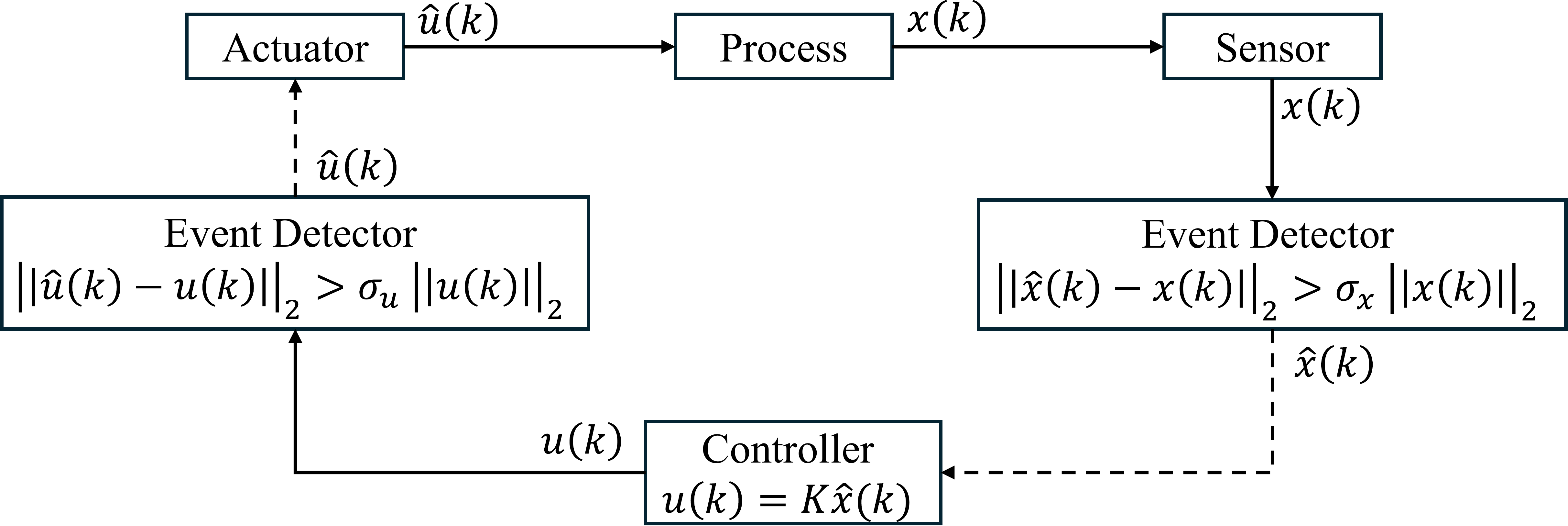}
    \caption{Block diagram of the information flow.}
    \label{Block diagram}
\end{figure}
\section{Preliminaries}\label{Preliminaries}
{\color{blue}
\subsection{Event-Triggered Control in NCSs}\label{Event-trigger control}
}
The block diagram of a typical NCS is shown in \cref{Block diagram}, where an event-triggered mechanism and a state-feedback control strategy are considered. The dashed lines represent communication channels. 

The block labeled ``Process'' corresponds to a discrete-time linear time-invariant (LTI) system described by
\begin{equation}\label{basic system}
    x(k+1) = A_* x(k) + B_* \hat{u}(k) + E d(k),
\end{equation}
where $x(k) \in \mathbb{R}^{n_x}$ denotes the system state, $\hat{u}(k) \in \mathbb{R}^{n_u}$ is the control input applied to the system and $d(k) \in \mathbb{R}^{n_d}$ represents the process noise. {\color{blue}The S2C and C2A channels thereby consist of $n_x$ and $n_u$ communication links for transmitting state and input components, respectively}. We make the standard assumption that the pair $(A_*, B_*)$ is stabilizable. Throughout this paper, the matrices $A_*$ and $B_*$ are assumed to be unknown.


The sensor measures the full state information at each sampling instant. Each sampled state is first processed by an event detector, which determines whether to transmit the sampled state to the controller via the S2C channel based on the following event-triggered condition:
\begin{equation}\label{event-trigger condition x}
    \|x(k)-\hat{x}(k)\|_2 > \sigma_x \|x(k)\|_2,
\end{equation}
where $\hat{x}(k)$ denotes the most recently transmitted state and $\sigma_x > 0$ is a user-specified threshold. In other words, the current state $x(k)$ is transmitted to the controller only if it deviates from $\hat{x}(k)$ by more than a fraction $\sigma_x$ of its norm. This mechanism reduces unnecessary transmissions when the system state changes only slightly, thereby alleviating the communication burden over the S2C channel. A larger $\sigma_x$ results in less frequent transmissions but may degrade control performance, and vice versa.

The control input $u(k)$ is computed by a state-feedback controller using the most recently received state:
\begin{equation}\label{state-feedback controller}
    u(k) = K \hat{x}(k),
\end{equation}
where $K$ is the controller gain to be designed. The computed control input $u(k)$ is then processed by another event detector, which decides whether to transmit $u(k)$ to the actuator via the C2A channel according to the following event-triggered condition:
\begin{equation}\label{event-trigger condition u}
    \|u(k)-\hat{u}(k)\|_2 > \sigma_u \|u(k)\|_2,
\end{equation}
where $\hat{u}(k)$ denotes the most recently transmitted control input, and $\sigma_u > 0$ is another user-specified threshold. If the condition is satisfied, $u(k)$ is transmitted to the actuator; otherwise, the actuator continues to use $\hat{u}(k)$.

For clarity, let $\{k^x_i\}_{i\geq 1}$ and $\{k^u_i\}_{i\geq 1}$ denote the sequences of transmission instants for the state and control input, respectively, with $k^x_1 = 0$ and $k^u_1 = 0$. The most recently transmitted state and control input at time $k$ are given by
\begin{equation}
    \begin{split}
        \hat{x}(k) &= x(k^x_i), \quad k \in \mathbb{Z}_{[k^x_i, k^x_{i+1})}, \\
        \hat{u}(k) &= u(k^u_i), \quad k \in \mathbb{Z}_{[k^u_i, k^u_{i+1})}.
    \end{split}
\end{equation}
The next transmission instants are determined by
\begin{equation}
\begin{split}
    k^x_{i+1} &= \min\left\{k > k^x_i \mid \|x(k) - \hat{x}(k)\|_2 > \sigma_x \|x(k)\|_2 \right\}, \\
    k^u_{i+1} &= \min\left\{k > k^u_i \mid \|u(k) - \hat{u}(k)\|_2 > \sigma_u \|u(k)\|_2 \right\}.
\end{split}
\end{equation}

Define the state and input errors as
\begin{equation}
\begin{split}
    e_x(k) &\triangleq x(k) - \hat{x}(k), \\
    e_u(k) &\triangleq u(k) - \hat{u}(k).
\end{split}
\end{equation}
By construction, the following inequalities hold for all $k \in \mathbb{N}$:
\begin{equation}\label{event-trigger condition}
\begin{split}
    \|e_x(k)\|_2 &\leq \sigma_x \|x(k)\|_2, \\
    \|e_u(k)\|_2 &\leq \sigma_u \|u(k)\|_2.
\end{split}
\end{equation}

{\color{blue}
\subsection{Sparse Control in NCSs}
The feedback control strategy in \eqref{state-feedback controller} typically constructs a feedback loop that connects $\hat{x}$ and $u$ within the controller. In the centralized setting, this interconnection is dense, meaning that each control input $u_i$ depends on the entire state vector $\hat{x}$. While such centralized controllers achieve optimal performance, they impose significant communication and computation burdens, especially in large-scale systems with limited resources. To mitigate these burdens, we aim to design sparse controllers, where each control input depends only on a subset of the state variables.  This sparsity is reflected in the  gain matrix $K$. Specifically, the feedback law \eqref{state-feedback controller} can be written as
\begin{equation}
[u(k)]_i = \sum_{j=1}^{n_x} K_{ij}[\hat{x}(k)]_j, \quad \forall\, 1 \leq i \leq n_u.
\end{equation}
To distinguish from the inter-device communication in the S2C and C2A channels, we refer to the internal data exchange within the controller as dataflow. Each nonzero entry $K_{ij}$ thereby represents a dataflow link from the $j$-th state component to the computation of the $i$-th control input, indicating that $[\hat{x}(k)]_j$ must be transmitted and processed in the corresponding computation unit. Conversely, if $K_{ij} = 0$, such a dataflow link is not required. Therefore, the sparsity pattern of $K$ directly characterizes the internal dataflow topology of the controller, where the number of nonzero entries determines the number of required dataflow links as well as the associated computational load. This motivates the use of $\|K\|_0$ as a measure of dataflow and computation complexity.}

{\color{blue}
\subsection{A Unified Cost Modeling Framework} \label{unified cost model}
The communication and computation cost (operational cost) of NCSs primarily arises from three sources: the communication cost over S2C channels, the communication cost over C2A channels, and the dataflow and computation cost within the controller. To capture these components in a unified manner, we consider the following operational cost model, which characterizes the average cost per time step:
\begin{equation*}
    J = n_x \alpha_x r_x + n_u \alpha_u r_u + (\alpha_m + \alpha_f)\|K\|_0 + \alpha_a (\|K\|_0-\#_{nzr}).
\end{equation*}
Here, $r_x$ and $r_u$ denote the transmission rates over the S2C and C2A channels, respectively. The positive coefficients $\alpha_x, \alpha_u, \alpha_f, \alpha_m, \alpha_a$ quantify the costs associated with one S2C link, one C2A link, one dataflow link, one multiplication, and one addition, respectively. Besides, $\#_{nzr}$ denotes the number of nonzero rows in $K$. Since the number of S2C and C2A links are $n_x$ and $n_u$, respectively, the first two terms represent the average communication cost. Moreover, $\|K\|_0$ corresponds to the number of dataflow links and multiplication operations, while $\|K\|_0 - \#_{nzr}$ equals the number of addition operations. Hence, the last two terms capture the internal dataflow and computation cost induced by the controller structure. This unified cost perspective highlights an inherent trade-off between ETC and SC, as $r_x$ and $r_u$ are implicitly determined by $\sigma_x$ and $\sigma_u$, respectively. This interplay motivates the co-design of ETC and SC to achieve efficient resource utilization in NCSs. Directly optimizing the above cost model is challenging for two reasons: first, the transmission rates $r_x$ and $r_u$ do not admit explicit analytical expressions in terms of the triggering parameters; second, even if such relationships are approximated, the resulting objective becomes highly nonconvex due to the implicit and nonlinear dependence between transmission rates and triggering thresholds. Therefore, we focus on the co-design problem under given event-triggered parameters, which serves as a tractable and essential step toward optimizing the overall operational cost.
}

\subsection{Problem Formulation}
This paper investigates the joint design of event-triggered parameters $\sigma_x, \sigma_u$ and a sparse controller $K$ for discrete-time linear systems with unknown $(A_*, B_*)$, using a data-driven approach. We assume access to a sequence of measured system states and control inputs corrupted by bounded measurement noise and process noise. Under these assumptions, we construct an uncertainty set $\Sigma$ containing all $(A, B)$ pairs that can possibly generate the measured data. The controller design must therefore be robust for all $(A, B) \in \Sigma$. Details of the data-driven framework are provided in \cref{Data Collection}. Before formalizing the problems, we introduce several key concepts essential for the subsequent analysis.

\begin{mydef}
    Consider system \eqref{basic system} with $d(k) = 0$. The closed-loop system under the event-triggered condition \eqref{event-trigger condition} and control law \eqref{state-feedback controller} is said to be \emph{asymptotically stable} if, for any initial state $x(0)$, $\lim_{k \to \infty} \|x(k)\|_2 = 0$.
\end{mydef}

\begin{mydef}
    For system \eqref{basic system} with $d(k) \neq 0$, the closed-loop system under the event-triggered condition \eqref{event-trigger condition} and control law \eqref{state-feedback controller} is said to be \emph{uniformly ultimately bounded (UUB)} in a compact set $\mathcal{S}$ if, for any initial state $x(0)$, there exists $T \geq 0$ such that $x(k) \in \mathcal{S}$ for all $k \geq T$.
\end{mydef}

\begin{mydef}
    Given a performance output $y$, the closed-loop system under the event-triggered condition \eqref{event-trigger condition} and control law \eqref{state-feedback controller} is said to have an $H_\infty$ norm less than $\gamma>0$ if, for zero initial state $x(0) = 0$ and any nonzero input $d \in l_2[0, +\infty)$, the output $y$ satisfies
    \begin{equation}
        \sum_{k=0}^\infty \|y(k)\|_2^2 \leq \gamma^2 \sum_{k=0}^\infty \|d(k)\|_2^2.
    \end{equation}
\end{mydef}

We now formally state the problems addressed in this paper.

\begin{myproblem}
    Given $\sigma_x, \sigma_u$, design the sparsest controller $K$ such that, for all $(A, B) \in \Sigma$, the closed-loop system with $d(k) = 0$ is asymptotically stable.
\end{myproblem}

\begin{myproblem}
    Given $\sigma_x, \sigma_u$ and an upper bound on $\|d(k)\|_2$, characterize the trade-off between the sparsity of the controller $K$ and the size of the UUB set $\mathcal{S}$, valid for all $(A, B) \in \Sigma$.
\end{myproblem}

\begin{myproblem}
    Given $\sigma_x, \sigma_u$, characterize the trade-off between the sparsity of the controller $K$ and the $H_\infty$ performance level $\gamma$, such that the $H_\infty$ norm of the closed-loop system is less than or equal to $\gamma$ for all $(A, B) \in \Sigma$.
\end{myproblem}

{\color{blue}The above formulations reveal an inherent coupling between communication and computation resources induced by the event-triggered mechanism and sparse control design, where event triggering affects transmission rates over the S2C and C2A channels, while controller sparsity determines internal computational and dataflow complexity.}

\section{Main Results}\label{Main Results}
Our main results consist of four parts: 1) the data collection mechanism and the construction of the uncertainty set $\Sigma$; 2) stability analysis in the absence of process noise; 3) UUB analysis with process noise; 4) $H_\infty$ performance analysis with process noise. Before presenting these results, we first introduce the lossy matrix S-lemma \cite{bisoffi2021trade} that will be frequently used in the rest of this paper.
\begin{mylem}[{\citep[Lemma 2]{bisoffi2021trade}}]\label{s procedure}
Let $T_0,\ldots,T_l\in \mathbb{R}^{(q+p)\times(q+p)}$ be symmetric matrices. If there exist scalars $\tau_1\ge0,\ldots,\tau_l\ge0$ such that $T_0-\sum_{i=1}^l\tau_iT_i\succeq0$, then \(\left[\begin{smallmatrix} I \\ Z \end{smallmatrix}\right]^\top T_0 \left[\begin{smallmatrix} I \\ Z \end{smallmatrix}\right] \succeq 0\) for all \( Z \in \mathbb{R}^{p \times q} \) such that \(\left[\begin{smallmatrix} I \\ Z \end{smallmatrix}\right]^\top T_i \left[\begin{smallmatrix} I \\ Z \end{smallmatrix}\right] \succeq 0\) for each \( i = 1, \ldots, \ell \).
\end{mylem}

\subsection{Data Collection}\label{Data Collection}
In this section, we present the data collection mechanism, which serves as the foundation for the subsequent analysis. Throughout this paper, we assume that $A_*$ and $B_*$ are unknown \textit{a priori}, and other system matrices are known. Since $(A_*, B_*)$ are unknown, direct control design is infeasible. In this work, we assume access to a sequence of measured system states and control inputs associated with \eqref{basic system}, both of which are subject to bounded measurement noise. For the purpose of data-driven analysis, \eqref{basic system} can be rewritten as 
\begin{equation}\label{data collection plant} x(k+1) = A_* x(k) + B_* u(k) + Ed(k), 
\end{equation} where the hat notation on $u$ is omitted to facilitate subsequent discussions. At each time step $k$, the measured state and input are given by
\begin{equation}
\begin{split}
    x_m(k) &= x(k) + \delta_x(k),\\
    u_m(k) &= u(k) + \delta_u(k),
\end{split}
\end{equation}
where the subscript `$m$' denotes measured quantities, and $\delta_x(k)$ and $\delta_u(k)$ are the state and input measurement noise, respectively.

We collect the measured data into the following matrices:
{\color{blue}
\begin{equation}\label{measured state and input}
\begin{aligned}
    X_m^+  &\triangleq \begin{bmatrix}x_m(1) &x_m(2) &\dots &x_m(T)\end{bmatrix} \in \mathbb{R}^{n_x \times T},\\
    X_m^- &\triangleq \begin{bmatrix}x_m(0) &x_m(1) &\dots &x_m(T-1)\end{bmatrix} \in \mathbb{R}^{n_x \times T},\\
    U_m^- &\triangleq \begin{bmatrix}u_m(0) &u_m(1) &\dots &u_m(T-1)\end{bmatrix} \in \mathbb{R}^{n_u \times T}.
\end{aligned}
\end{equation}}
We impose the following assumptions on the process and measurement noise during data collection.
\begin{myass}\label{process measurement noise bound assumption}
    The process noise $d(k)$ and measurement noise $\delta_x(k)$ and $\delta_u(k)$ are bounded as follows:
    \begin{subequations}
    \begin{align}
        &\|d(k)\|_2 \leq \epsilon_d, \quad \forall k = 0, \ldots, T-1,\label{process noise bound} \\
        &\|\delta_x(k)\|_2 \leq \epsilon_x, \quad \forall k = 0, \ldots, T, \\
        &\|\delta_u(k)\|_2 \leq \epsilon_u, \quad \forall k = 0, \ldots, T-1,
    \end{align}
    \end{subequations}
    for some known constants $\epsilon_d, \epsilon_x, \epsilon_u > 0$.
\end{myass}

By substituting the measured data into \eqref{data collection plant}, we obtain
\begin{equation}
\begin{split}
    x_m(k+1) - \delta_x(k+1) =\ &A_* (x_m(k) - \delta_x(k)) \\
    &+ B_* (u_m(k) - \delta_u(k)) + Ed(k).
\end{split}
\end{equation}
For notational simplicity, we drop the subscript $*$ and refer to all $(A, B)$ that can generate $(x_m(k), u_m(k), x_m(k+1))$ for some admissible noise sequences as \emph{all possible $(A, B)$ consistent with the data}. Rearranging the terms yields
\begin{equation}\label{compact way}
    \begin{bmatrix}
        I\\
        A^\top\\
        B^\top
    \end{bmatrix}^\top
    \begin{bmatrix}
        x_m(k+1)\\
        -x_m(k)\\
        -u_m(k)
    \end{bmatrix}
    =
    \begin{bmatrix}
        I\\
        A^\top\\
        B^\top
    \end{bmatrix}^\top
    G
    \begin{bmatrix}
        d(k)\\
        \delta_x(k)\\
        \delta_x(k+1)\\
        \delta_u(k)
    \end{bmatrix},
\end{equation}
where $G = \begin{bmatrix}
        \begin{smallmatrix}
        E & 0 & I & 0\\
        0 & -I & 0 & 0\\
        0 & 0 & 0 & -I
        \end{smallmatrix}
     \end{bmatrix}$. Following the approach in \cite{van2020noisy,van2023quadratic}, we characterize the set of all $(A, B)$ consistent with the data using a quadratic matrix inequality (QMI). The following lemma provides a useful result that facilitates this derivation.
{\color{blue}
\begin{mylem}\label{lemma:joint bound of process noise and measurement noise}
    If \cref{process measurement noise bound assumption} holds, we have the inequality:
    \begin{equation}\label{equation:joint bound of process noise and measurement noise}
        \begin{bmatrix}
            d(k)\\
            \delta_x(k)\\
            \delta_x(k+1)\\
            \delta_u(k)
        \end{bmatrix}
        \begin{bmatrix}
            d(k)\\
            \delta_x(k)\\
            \delta_x(k+1)\\
            \delta_u(k)
        \end{bmatrix}^\top
        \preceq \Pi,
    \end{equation}
    where \begin{equation*}\begin{split}
        \Pi = \mathrm{blkdiag}(&4\epsilon_d^2 I,\, 4\epsilon_x^2 I,\, 4\epsilon_x^2 I,\, 4\epsilon_u^2 I ).
        \end{split}\end{equation*}
    
\end{mylem}
\noindent\emph{Proof:} See Appendix~\ref{proof of lemma:joint bound of process noise and measurement noise}.}

Applying \cref{lemma:joint bound of process noise and measurement noise} to \eqref{compact way}, we obtain a QMI that describes all $(A, B)$ consistent with the data. Specifically, for each $k$, we have
\begin{equation}
    \begin{bmatrix}
        I\\
        A^\top\\
        B^\top
    \end{bmatrix}^\top
    \Psi_k
    \begin{bmatrix}
        I\\
        A^\top\\
        B^\top
    \end{bmatrix}
    \succeq 0,
\end{equation}
where
\begin{equation}
    \Psi_k = G \Pi G^\top - \begin{bmatrix}
        x_m(k+1)\\
        -x_m(k)\\
        -u_m(k)
    \end{bmatrix}
    \begin{bmatrix}
        x_m(k+1)\\
        -x_m(k)\\
        -u_m(k)
    \end{bmatrix}^\top.
\end{equation}
The set of $(A, B)$ consistent with the entire measured trajectory is thus given by
\begin{equation}
    \Sigma = \bigcap_{k=0}^{T-1} \left\{ (A, B) \,\middle|\, \begin{bmatrix}
        I\\
        A^\top\\
        B^\top
    \end{bmatrix}^\top
    \Psi_k
    \begin{bmatrix}
        I\\
        A^\top\\
        B^\top
    \end{bmatrix}
    \succeq 0 \right\}.
\end{equation}
{\color{blue}
Next, we impose an assumption on the data matrix in \eqref{measured state and input}.
\begin{myass}\label{data richness assumption}
    $\begin{bmatrix} X_m^- \\ U_m^- \end{bmatrix}\begin{bmatrix} X_m^- \\ U_m^- \end{bmatrix}^\top \succeq \gamma I$ for some $\gamma > 0$.
\end{myass} 
\cref{data richness assumption} ensures that the collected data are sufficiently informative, and it can be verified directly from the measured data. In particular, this condition implies that $T \ge n_x + n_u$, and guarantees that $\Sigma$ is bounded when the measurement noise is sufficiently small. This result is formalized as follows.
\begin{mylem}\label{boundedness of uncertainty set}
    $\Sigma$ is bounded if \cref{data richness assumption} holds, and $\epsilon_x$ and $\epsilon_u$ satisfy $T \cdot \max(4\epsilon_x^2, 4\epsilon_u^2) \leq \gamma$.
\end{mylem}
\noindent\emph{Proof:} See Appendix~\ref{proof of lemma:boundedness of uncertainty set}.

With the uncertainty set $\Sigma$ established, we can now proceed to solve \textit{Problems 1--3}.
}

\subsection{Stability Analysis in the Absence of Process Noise}
In this section, we address \textit{Problem 1}: given the event-triggered parameters $\sigma_x$ and $ \sigma_u$, our objective is to design a controller $K$ with maximal sparsity such that the closed-loop system is asymptotically stable for all $(A, B) \in \Sigma$. {\color{blue}From a networked control perspective, the parameters $\sigma_x$ and $\sigma_u$ determine the transmission behavior over the S2C and C2A channels, respectively, while the sparsity of $K$ reflects the computational and dataflow complexity of the controller}. Throughout this section, we assume the system is free of process noise, i.e., $d(k) = 0$.

Based on the event-triggered mechanism described in \cref{Event-trigger control}, the closed-loop system can be written as
\begin{equation}\label{closed-loop system problem 1}
\begin{split}
    x(k+1) &= A x(k) + B \hat{u}(k) \\
           &= A x(k) + B (u(k) - e_u(k)) \\
           &= A x(k) + B K (x(k) - e_x(k)) - B e_u(k) \\
           &= A x(k) + B K x(k) - B K e_x(k) - B e_u(k),
\end{split}
\end{equation}
where $x(k)$, $e_x(k)$, and $e_u(k)$ satisfy $\|e_x(k)\|_2 \leq \sigma_x \|x(k)\|_2$ and $\|e_u(k)\|_2 \leq \sigma_u \|u(k)\|_2$. {\color{blue}Here, $e_x(k)$ and $e_u(k)$ represent the errors induced by event-triggered transmissions, capturing the effect of reduced communication over the S2C and C2A channels}. These inequalities can be equivalently expressed as QMIs in terms of $\left[\begin{smallmatrix} x^\top & e_x^\top(k) & e_u^\top(k) \end{smallmatrix}\right]^\top$. Specifically, $\|e_x(k)\|_2 \leq \sigma_x \|x(k)\|_2$ is equivalent to
\begin{equation}\label{event-trigger QMI 1}
    \begin{bmatrix}
        x(k)\\
        e_x(k)\\
        e_u(k)
    \end{bmatrix}^\top
    \begin{bmatrix}
        \sigma_x^2 I & 0 & 0 \\
        0 & -I & 0 \\
        0 & 0 & 0
    \end{bmatrix}
    \begin{bmatrix}
        x(k)\\
        e_x(k)\\
        e_u(k)
    \end{bmatrix} \geq 0.
\end{equation}
Similarly, $\|e_u(k)\|_2 \leq \sigma_u \|u(k)\|_2$ can be rewritten as
\begin{equation}\label{event-trigger QMI 2}
    \begin{bmatrix}
        x(k)\\
        e_x(k)\\
        e_u(k)
    \end{bmatrix}^\top
    \begin{bmatrix}
        \begin{bmatrix} I \\ -I \end{bmatrix} \sigma_u^2 K^\top K \begin{bmatrix} I \\ -I \end{bmatrix}^\top & \begin{bmatrix} 0 \\ 0 \end{bmatrix} \\
        \begin{bmatrix} 0 & 0 \end{bmatrix} & -I
    \end{bmatrix}
    \begin{bmatrix}
        x(k)\\
        e_x(k)\\
        e_u(k)
    \end{bmatrix} \geq 0.
\end{equation}
We construct a Lyapunov function $V(k) = x^\top(k) P x(k)$ with $P \succ 0$. To ensure asymptotic stability, the Lyapunov function must satisfy $V(k+1) \le \beta V(k), \beta\in (0,1)$ for all $k \in \mathbb{N}$, i.e.,
\begin{equation}\label{lyapunov function decrease}
\begin{split}
    &(A x(k) + B K x(k) - B K e_x(k) - B e_u(k))^\top P (A x(k) \\
    &+ B K x(k) - B K e_x(k) - B e_u(k)) \le \beta x^\top(k) P x(k).
\end{split}
\end{equation}
{\color{blue}This condition ensures closed-loop stability despite the presence of transmission-induced errors arising from limited communication}. Note that the specific values of $A$, $B$, $x(k)$, $e_x(k)$, and $e_u(k)$ are unknown; the only information available is that $(A,B)\in\Sigma$ and $\left[\begin{smallmatrix}
    x^\top &e_x^\top(k) &e_u^\top(k)
\end{smallmatrix}\right]^\top$ satisfies \eqref{event-trigger QMI 1} and \eqref{event-trigger QMI 2}. Therefore, given $\sigma_x$ and $\sigma_u$, our objective is to design the sparsest controller $K$ such that \eqref{lyapunov function decrease} holds for $\forall(A,B)\in\Sigma$ and $\forall\left[\begin{smallmatrix}
    x^\top &e_x^\top(k) &e_u^\top(k)
\end{smallmatrix}\right]^\top$ satisfying \eqref{event-trigger QMI 1} and \eqref{event-trigger QMI 2}. Before addressing the sparsity aspect, it is natural to first ask whether there exists any feasible controller $K$ (not necessarily sparse) that achieves this objective. {\color{blue}The following result provides a sufficient condition, based on the measured data and the prescribed event-triggered parameters specifying the transmission rates over the S2C and C2A channels, for the existence of a stabilizing controller. This result also serves as a foundation for subsequently incorporating sparsity considerations, which are directly related to computational complexity.} {\color{blue}Before presenting the result, we define $\mathcal{T}(X) \triangleq \mathcal{B}_4(X)-\sum_{i=0}^{T-1}\theta_i\Psi_i$, $\Bar{\Omega}_i \triangleq 2\Bar{P}-\Bar{\alpha}_i I $, $\Bar{\Theta}_i \triangleq G^\top + G - \Bar{\alpha}_i I $, and $\Theta_i \triangleq G^\top + G - \alpha_i^{-1} I$ throughout this paper, where $i\in \mathbb{Z}$.}

\begin{mythr}\label{stability without noise theorem}
Consider system \eqref{basic system} with $d(k)=0$. Given the event-triggered parameters $\sigma_x,\sigma_u$ in \eqref{event-trigger condition} and $\beta\in (0,1)$, if there exist matrices $\Bar{P}\in\mathbb{R}^{n_x\times n_x}$, $L\in\mathbb{R}^{n_u\times n_x}$, $G\in\mathbb{R}^{n_u\times n_u}$, and scalars $\Bar{\alpha}_1 \geq 0$, $\Bar{\alpha}_2 \geq 0$, $\{\theta_i\}_{i=0,1,\ldots,T-1}$ such that the following semidefinite program (SDP) is feasible:

\begin{subequations}\label{stability without noise theorem eq}
    \begin{align}
    &\begin{aligned}
        &\begin{bmatrix}
            \mathcal{T}(\Bar{P}) &\mathcal{B}_1(\Bar{P},L,-G) &0\\
            * &\mathrm{blkdiag}(\beta \Bar{P}, \Bar{\Omega}_1, \Bar{\Theta}_2) &\mathcal{B}_2(L^\top,\Bar{P})\\
            * &* &\mathrm{blkdiag}(\frac{\Bar{\alpha}_2}{\sigma_u^2}I, \frac{\Bar{\alpha}_1}{\sigma_x^2}I)
        \end{bmatrix}\succeq0,
        \end{aligned} \label{stability without noise theorem eq1}\\
    & \Bar{P}\in \mathbb{S}^{n_x}_{++}, \quad \Bar{\alpha}_1\ge0, \quad \Bar{\alpha}_2\ge0, \quad \theta_0\ge0, \ldots, \theta_{T-1}\ge0, \label{stability without noise theorem eq2}
    \end{align}
\end{subequations}
then the closed-loop system \eqref{closed-loop system problem 1} is asymptotically stable under the controller $K = L\Bar{P}^{-1}$ with Lyapunov matrix $P = \Bar{P}^{-1}$.
\end{mythr}
\textit{Proof}: See Appendix~\ref{proof of stability without noise theorem}.

{\color{blue}This result shows that closed-loop stability can be guaranteed under limited communication induced by event-triggering, without yet imposing sparsity constraints on the controller. Our objective, however, is to design the sparsest such controller, thereby reducing the computational and dataflow complexity while preserving stability under the given communication constraints.} According to the change of variables $K = L \Bar{P}^{-1}$ in \cref{stability without noise theorem}, the corresponding optimization problem is
\begin{equation}
    \begin{aligned}
        &\mathop{\min}_{L,\Bar{P}} \|L \Bar{P}^{-1}\|_0 \qquad \text{s.t.} \\
        &\eqref{stability without noise theorem eq1},\ \eqref{stability without noise theorem eq2}.
    \end{aligned}
\end{equation}
However, minimizing the objective function $\|L \Bar{P}^{-1}\|_0$ is NP-hard, and the dependence of $K$ on $L$ and $\Bar{P}$ complicates direct optimization. To address this, we reformulate the problem to express $K$ explicitly, avoiding the change of variables. Starting from \eqref{temp1 in proof of thr1}, which is directly implied by \eqref{stability without noise theorem eq1}, and applying only the inequality $\alpha_1 P^{-1} P^{-1} \succeq 2P^{-1} - \alpha_1^{-1} I$, we obtain

\begin{equation}\label{temp1}
    \begin{aligned}
        &\scalebox{0.97}{$\begin{bmatrix}
            \mathcal{T}(P^{-1}) &\mathcal{B}_1(P^{-1},L,-G) &0\\
            * &\Delta &\mathcal{B}_2(L^\top,P^{-1})\\
            * &* &\mathrm{blkdiag}(\frac{1}{\alpha_2\sigma_u^2}I, \frac{1}{\alpha_1\sigma_x^2}I)
        \end{bmatrix}\succeq0, $}\\
        & \Delta = \mathrm{blkdiag}(\beta P^{-1}, \alpha_1(P^{-1})^2, \Theta_2).
        \end{aligned} 
\end{equation}
Next, we substitute $L =K\Bar{P}=K P^{-1}$, pre- and post-multiply \eqref{temp1} by $\mathrm{blkdiag}(I, I, I, P, P, I, I, I)$ and its transpose, and apply the Schur complement. This yields a formulation in which $K$ appears explicitly. Incorporating the sparsity-promoting objective $\|K\|_0$, the resulting optimization problem can be summarized as follows:

\begin{subequations}\label{temp3}
    \begin{align}
        &\mathop{\min}_{P,K,G,\alpha_1,\Bar{\alpha}_2,\{\theta_i\}_{i=0,1,\ldots,T-1}} \|K\|_0 \qquad \text{s.t.} \\
        &\scalebox{0.97}{$\begin{bmatrix}
            \mathcal{T}(P^{-1}) &\mathcal{B}_1(I,K,-G) &0\\
            * &\mathrm{blkdiag}(\beta P-\alpha_1 \sigma_x^2 I, \alpha_1 I, \Bar{\Theta}_2) &\mathcal{B}_3(K^\top)\\
            * &* &\frac{\Bar{\alpha}_2}{\sigma_u^2}I
        \end{bmatrix}\succeq0$},\label{temp3 eq1}\\
    &P\in \mathbb{S}^{n_x}_{++}, \alpha_1\ge0, \Bar{\alpha}_2\ge0, \theta_0\ge0, \theta_1\ge0,\ldots, \theta_{T-1}\ge0.\label{temp3 eq2}
    \end{align}
\end{subequations}
However, directly solving Problem \eqref{temp3} presents two main challenges. First, the objective function $\|K\|_0$ is non-convex and NP-hard to optimize. Second, the constraint \eqref{temp3 eq1} is non-convex due to the presence of the matrix inverse $P^{-1}$. To address the first challenge, we adopt the reweighted $\ell_1$ minimization approach proposed by Cand{\`e}s et al. \cite{candes2008enhancing}, which provides a tractable and effective approximation to the $\ell_0$ norm. Unlike standard $\ell_1$ minimization, which applies uniform weights and may overly penalize large entries, the reweighted $\ell_1$ method iteratively updates the weights so that smaller entries are penalized more heavily, thereby promoting sparsity more effectively and yielding solutions that better approximate true $\ell_0$ minimization.
To address the second challenge, we linearize the non-convex term $P^{-1}$ around a given point $\Tilde{P}$ and adopt the inequality

\begin{equation}\label{linearization of P inverse}
    P^{-1}\succeq \Tilde{P}^{-1} - \Tilde{P}^{-1}(P-\Tilde{P})\Tilde{P}^{-1},
\end{equation}
which arises from the convexity of matrix inversion in the Loewner order \cite{bhatia2013matrix}. With these two modifications, we can reformulate Problem \eqref{temp3} as follows:

\begin{subequations}\label{real optimization problem of problem 1}
    \begin{align}
        &\mathop{\min}_{Y,P,K,G,\alpha_1,\Bar{\alpha}_2,\{\theta_i\}_{i=0,1,\ldots,T-1}} \mathrm{Tr}(W^\top K^{|\cdot|}) \qquad \text{s.t.} \nonumber\\
        &\begin{bmatrix}
            \mathcal{T}(Y) &\mathcal{B}_1(I,K,-G) &0\\
            * &\mathrm{blkdiag}(\beta P-\alpha_1 \sigma_x^2 I, \alpha_1 I, \Bar{\Theta}_2) &\mathcal{B}_3(K^\top)\\
            * &* &\frac{\Bar{\alpha}_2}{\sigma_u^2}I
        \end{bmatrix}\succeq0,\label{real optimization problem of problem 1 eq1}\\
    &Y \preceq \Tilde{P}^{-1} - \Tilde{P}^{-1}(P-\Tilde{P})\Tilde{P}^{-1},\label{real optimization problem of problem 1 eq2}\\
        &P\in \mathbb{S}^{n_x}_{++},\quad Y\in \mathbb{S}^{n_x}_{++},\\
        &\alpha_1\ge0,\quad \Bar{\alpha}_2\ge0,\quad \theta_0\ge0,\, \ldots,\, \theta_{T-1}\ge0,
    \end{align}
\end{subequations}
where $W$ is a given weight matrix. Note that due to the inequality \eqref{linearization of P inverse}, the feasible region of Problem \eqref{real optimization problem of problem 1} is contained within that of Problem \eqref{temp3}, which may cause infeasibility if a poor linearization point $\Tilde{P}$ is chosen. The following result provides guidance on how to select $\Tilde{P}$ by establishing the relationship among Problem \eqref{stability without noise theorem eq}, Problem \eqref{temp3}, and Problem~\eqref{real optimization problem of problem 1}.
\begin{mypro}\label{relationship between three problems}
    Given a feasible solution $(Y,P,K,G,\alpha_1,\Bar{\alpha}_2,\{\theta_i\}_{i=0,1,\ldots,T-1})$ to Problem \eqref{real optimization problem of problem 1}, then $(P,K,G,\alpha_1,\Bar{\alpha}_2,\{\theta_i\}_{i=0,1,\ldots,T-1})$ is feasible for Problem \eqref{temp3}. Conversely, given a feasible solution $(P,K,G,\alpha_1,\Bar{\alpha}_2,\{\theta_i\}_{i=0,1,\ldots,T-1})$ to Problem \eqref{temp3}, setting $\Tilde{P} = P$ ensures that Problem \eqref{real optimization problem of problem 1} is feasible. Furthermore, given a feasible solution $(\Bar{P},L,G,\Bar{\alpha}_1,\Bar{\alpha}_2,\{\theta_i\}_{i=0,1,\ldots,T-1})$ to Problem \eqref{stability without noise theorem eq}, setting $\Tilde{P} = \Bar{P}^{-1}$ guarantees that Problem~\eqref{real optimization problem of problem 1} is feasible.
\end{mypro}
\textit{Proof}: See Appendix~\ref{proof of relationship between three problems}.

\begin{algorithm}[t]
    \caption{Iterative algorithm to solve \textit{Problem 1}.}
    \textbf{Output:} $K^*$;\\
    Solve Problem \eqref{stability without noise theorem eq} and initialize $P_0 = \Bar{P}^{-1}$, $W_0 = \mathbf{1}_{n_u \times n_x}$, $\epsilon_r > 0$, $k = 0$;\\
    \Repeat{\text{convergence or the maximum iteration reached}}{
        - Solve Problem \eqref{real optimization problem of problem 1} with $\Tilde{P} = P_k$, $W = W_k$;\\
        - Obtain the solutions $K_{k+1}, P_{k+1}$;\\
        - $W_{k+1} = (K_{k+1}^{|\cdot|} + \epsilon_r \cdot \mathbf{1}_{n_u \times n_x})^{\circ -1}$;\\
        - $k = k + 1$;
    }
    \Return $K_k$.
    \label{algorithm 1}
\end{algorithm} 
Proposition~\ref{relationship between three problems} motivates an iterative algorithm to solve Problem \eqref{temp3} by repeatedly solving Problem \eqref{real optimization problem of problem 1}, initialized with $\Tilde{P}$ obtained from Problem \eqref{stability without noise theorem eq}. The complete procedure is summarized in Algorithm~\ref{algorithm 1}\footnote{$\epsilon_r$ is a small positive constant (e.g., $10^{-6}$) to avoid division by zero.}. {\color{blue}The resulting procedure can be interpreted as a co-design strategy that explicitly  balances communication efficiency (through event-triggering) and computational complexity (through controller sparsity), while ensuring closed-loop stability.}

\subsection{UUB Analysis with Bounded Process Noise}
In this section, we address \textit{Problem 2}. When process noise is present, i.e., $d(k) \neq 0$, asymptotic stability is no longer achievable in general. Instead, the objective shifts to ensuring that the closed-loop system is UUB for all $(A, B) \in \Sigma$. In this context, the size of the UUB set $\mathcal{S}$, as defined in \eqref{bounded set}, serves as a measure of closed-loop performance. Our objective in this section is to investigate the trade-off between the size of $\mathcal{S}$ and the sparsity of the controller $K$, given the event-triggered parameters $\sigma_x$ and $\sigma_u$. {\color{blue}From a networked control perspective, the combined effect of event-triggered communication (characterized by $\sigma_x$ and $\sigma_u$), external disturbances, and controller sparsity introduces a fundamentally different trade-off compared to the noise-free case. Specifically, reduced communication increases transmission-induced errors, disturbances enlarge the ultimate bound, and sparsity further limits achievable performance due to reduced control flexibility.}

Based on the derivation in the previous section, the closed-loop system can be written as
\begin{equation}\label{closed-loop system problem 2}
    \scalebox{0.96}{$
    \begin{aligned}
        x(k+1) = A x(k) + B K x(k) - B K e_x(k) - B e_u(k) + E d(k),
    \end{aligned}
    $}
\end{equation}
where $d(k)$ denotes the process noise, which is assumed to be bounded such that $\|d(k)\|_2 \leq \epsilon_d$ for all $k \geq 0$. This assumption is consistent with the bound specified in \eqref{process noise bound} during data collection. The variables $x(k)$, $e_x(k)$, and $e_u(k)$ satisfy \eqref{event-trigger QMI 1} and \eqref{event-trigger QMI 2}. We consider the Lyapunov function $V(k) = x^\top(k) P x(k)$ with $P \succ 0$, and define the UUB set as
\begin{equation}\label{bounded set}
    \mathcal{S} \triangleq \left\{ x \in \mathbb{R}^{n_x} : x^\top P x \leq 1 \right\}.
\end{equation}

Clearly, the set $\mathcal{S}$ depends on the choice of $P$, which is determined by the controller $K$. The validity of the UUB set requires the following two conditions: (i) $V(k+1) \le \beta V(k), \beta\in (0,1)$ whenever $V(k) \geq 1$; and (ii) $V(k+1) \leq 1$ whenever $V(k) \leq 1$.

Since $A$, $B$, $x(k)$, $e_x(k)$, $e_u(k)$, and $d(k)$ are unknown, we can guarantee that $\mathcal{S}$ is a valid UUB set by ensuring conditions (i) and (ii) hold for $\forall(A, B) \in \Sigma$, $\forall d(k)$ satisfying $\|d(k)\|_2 \leq \epsilon_d$, and $\forall\left[\begin{smallmatrix} x^\top & e_x^\top(k) & e_u^\top(k) \end{smallmatrix}\right]^\top$ satisfying \eqref{event-trigger QMI 1} and \eqref{event-trigger QMI 2}. {\color{blue}Before incorporating sparsity into the design, we first characterize the existence of a UUB-stabilizing controller based solely on data and prescribed triggering parameters. This result serves as the foundation for the subsequent co-design with sparsity constraints.
}

\begin{mythr}\label{stability with noise theorem}
    Consider system \eqref{basic system} with $d(k)$ satisfying $\|d(k)\|_2\leq \epsilon_d$ for all $k\ge0$. Given the event-triggered parameters $\sigma_x,\sigma_u$ in \eqref{event-trigger condition} and $\beta\in (0,1)$, suppose there exist matrices $\Bar{P}\in\mathbb{R}^{n_x\times n_x}$, $L\in\mathbb{R}^{n_u\times n_x}$, $G\in\mathbb{R}^{n_u\times n_u}$, and scalars $\alpha_1$, $\Bar{\alpha}_2$, $\Bar{\alpha}_3$, $\Bar{\alpha}_4$, $\{\theta_i\}_{i=0,1,\ldots,T-1}$ such that the following holds:

\begin{subequations}\label{stability with noise theorem eq}
    \begin{align}
    &\begin{aligned}
        &\scalebox{0.95}{$\begin{bmatrix}
            \mathcal{T}(\Bar{P}-\Bar{\alpha}_2EE^\top) &\mathcal{B}_1(\Bar{P},L,-G) &0\\
            * & \Delta &\mathcal{B}_2(L^\top,\Bar{P})\\
            * &* &\mathrm{blkdiag}(\frac{\Bar{\alpha}_4}{\sigma_u^2}I, \frac{\Bar{\alpha}_3}{\sigma_x^2}I)
        \end{bmatrix}\succeq0,$}\\
        & \Delta = \mathrm{blkdiag}((\beta-\alpha_1) \Bar{P}, \Bar{\Omega}_3, \Bar{\Theta}_4),
        \end{aligned} \label{stability with noise theorem eq1}\\
    & \Bar{P}\in \mathbb{S}^{n_x}_{++}, 1\ge\alpha_1\ge0, \Bar{\alpha}_2\ge0, \Bar{\alpha}_3\ge0, \Bar{\alpha}_4\ge0, \Bar{\alpha}_2\ge\frac{\epsilon_d^2}{\alpha_1},  \label{stability with noise theorem eq2}\\ 
        &\theta_0\ge0, \theta_1\ge0,\ldots, \theta_{T-1}\ge0,\label{stability with noise theorem eq3}
    \end{align}
\end{subequations}
then the closed-loop system \eqref{closed-loop system problem 2} is UUB in the set $\mathcal{S}$ defined in \eqref{bounded set}, under the controller $K=L\Bar{P}^{-1}$ and Lyapunov matrix $P=\Bar{P}^{-1}$.
\end{mythr}
\textit{Proof}: See Appendix~\ref{proof of stability with noise theorem}.

Problem \eqref{stability with noise theorem eq} is not a standard semidefinite program (SDP) due to the presence of $(\beta-\alpha_1)\Bar{P}$ and $\Bar{\alpha}_2\ge\frac{\epsilon_d^2}{\alpha_1}$. To address this, we propose fixing $\alpha_1$ throughout this section, which, while potentially introducing conservativeness, renders Problem \eqref{stability with noise theorem eq} convex and amenable to existing SDP solvers. The choice of $\alpha_1$ directly influences the size of the UUB set $\mathcal{S}$, and thus the closed-loop performance. We now discuss how to select an appropriate value for $\alpha_1$.

Recall the definition of the UUB set $\mathcal{S}$ in \eqref{bounded set}. Intuitively, the size of $\mathcal{S}$ decreases as $P$ increases, and our objective is to minimize the size of $\mathcal{S}$. From the structure of \eqref{stability with noise theorem eq1} and $P^{-1}-\frac{\epsilon_d^2}{\alpha_1}EE^\top \succeq P^{-1}-\Bar{\alpha}_2EE^\top$, it follows that a small $\alpha_1$ can possibly lead to small $P$, and thus degrade performance by enlarging $\mathcal{S}$. The following result characterizes the feasibility of Problem \eqref{stability with noise theorem eq} as a function of $\alpha_1$.

\begin{mypro}\label{monotonicity pro}
    If Problem \eqref{stability with noise theorem eq} is feasible for $\alpha_1 = c$, then it is also feasible for all $\alpha_1 < c$.
\end{mypro}
\textit{Proof}: See Appendix~\ref{proof of monotonicity pro}.
\begin{algorithm}[t]
    \caption{Iterative algorithm to solve \textit{Problem 2}.}
\textbf{Output:} $K^*,P^*$;\\
    Find the maximal $\alpha_1\in[0,1]$ such that Problem \eqref{stability with noise theorem eq} is feasible, and fix this $\alpha_1$ throughout this algorithm;\\
    With such a $\alpha_1$, solve Problem \eqref{stability with noise theorem eq} and initialize $P_0=\Bar{P}^{-1}, W_0=\mathbf{1}_{n_u\times n_x}, \epsilon_r>0,\lambda>0,k=0$;\\
    \Repeat{\text{convergence or the maximum iteration reached\vspace{5pt}}}{
    - Solve Problem \eqref{real optimization problem of problem 2} with $\Tilde{P} = P_k, W=W_k$;\\
    - Obtain the solution $K_{k+1}, P_{k+1}$;\\
    - $W_{k+1} = (K_{k+1}^{|\cdot|}+\epsilon_r \cdot \mathbf{1}_{n_u\times n_x})^{\circ -1}$;\\
    - $k = k + 1$;
    }
\Return $K_k,P_k$.
\label{algorithm 2}
\end{algorithm} 
\begin{mycor}
    There exists a scalar $c_{\max} \in [0,1]$ such that Problem \eqref{stability with noise theorem eq} is feasible for all $\alpha_1 \in (0, c_{\max}]$ and infeasible for all $\alpha_1 \in (c_{\max}, 1]$.
\end{mycor}
\textit{Proof}: This result follows directly from Proposition~\ref{monotonicity pro}. $\hfill \square$

Accordingly, we select $\alpha_1$ as the largest value in $[0,1]$ for which Problem \eqref{stability with noise theorem eq} remains feasible.

To quantify the size of $\mathcal{S}$, we use the sum of the squares of its semi-axes, which can be expressed as $\mathrm{Tr}(P^{-1})$ \cite{kurzhanski1997ellipsoidal}. We are now prepared to investigate the trade-off between the size of $\mathcal{S}$ and the sparsity of the controller $K$. In the following, we present the final optimization problem that motivates an iterative algorithm to solve \textit{Problem 2}, followed by a detailed explanation.

\begin{subequations}\label{real optimization problem of problem 2}
    \begin{align}
        &\mathop{\min}_{Y,P,Q,K,G,\Bar{\alpha}_2,\alpha_3,\Bar{\alpha}_4,\{\theta_i\}_{i=0,1,\ldots,T-1}} \mathrm{Tr}(Q)+\lambda\mathrm{Tr}(W^\top K^{|\cdot|})\  \text{s.t.} \nonumber\\
    &\begin{aligned}
        &\begin{bmatrix}
            \mathcal{T}(Y-\Bar{\alpha}_2 EE^\top) &\mathcal{B}_1(I,K,-G) &0\\
            * & \Delta&\mathcal{B}_3(K^\top)\\
            * &* &\frac{\Bar{\alpha}_4}{\sigma_u^2}I
        \end{bmatrix}\succeq0,\\
        & \Delta = \mathrm{blkdiag}((\beta-\alpha_1)P-\alpha_3 \sigma_x^2 I, \alpha_3 I, \Bar{\Theta}_4)
        \end{aligned} \label{real optimization problem of problem 2 eq1}\\
    &Y \preceq \Tilde{P}^{-1} - \Tilde{P}^{-1}(P-\Tilde{P})\Tilde{P}^{-1}, \label{real optimization problem of problem 2 eq2}\\
        &\begin{bmatrix}
            Q &I\\
            I &P
            \end{bmatrix}\succeq 0,\label{real optimization problem of problem 2 eq3}\\
        &P\in \mathbb{S}^{n_x}_{++}, Q\in \mathbb{S}^{n_x}_{++}, Y\in \mathbb{S}^{n_x}_{++},\\
        &\Bar{\alpha}_2\ge0, \alpha_3\ge0,\Bar{\alpha}_4\ge0, \Bar{\alpha}_2\ge\frac{\epsilon_d^2}{\alpha_1},\\ 
        &\theta_0\ge0,\theta_1\ge0,\ldots, \theta_{T-1}\ge0.
    \end{align}
\end{subequations}
{\color{blue}From a networked control perspective, this formulation explicitly captures a three-way trade-off among communication efficiency (through $\sigma_x$ and $\sigma_u$), disturbance attenuation (through $\epsilon_d$), and controller sparsity. In particular,} The parameter $\lambda$ in the objective function is a positive scalar that trades off between the size of the UUB set $\mathcal{S}$ and the sparsity of the controller $K$; increasing $\lambda$ promotes sparser controller solutions at the potential expense of performance. The constraint in \eqref{real optimization problem of problem 2 eq1} is derived from \eqref{stability with noise theorem eq1} following a procedure analogous to that used for \eqref{real optimization problem of problem 1 eq1}, while \eqref{real optimization problem of problem 2 eq2} is identical to \eqref{real optimization problem of problem 1 eq2}. In \eqref{real optimization problem of problem 2 eq3}, the auxiliary variable $Q$ is introduced to facilitate the quantification of the size of $\mathcal{S}$. Before presenting the final algorithm, we first establishes the relationship between Problem \eqref{real optimization problem of problem 2} and Problem \eqref{stability with noise theorem eq}.
    \begin{mypro}\label{relationship between two problems}
        Given a feasible solution $(Y,P,Q,K,G,\Bar{\alpha}_2,\alpha_3,\Bar{\alpha}_4,\{\theta_i\}_{i=0,1,\ldots,T-1})$ to Problem~\eqref{real optimization problem of problem 2}, setting $\Tilde{P} = P$ ensures that Problem~\eqref{real optimization problem of problem 2} remains feasible. Furthermore, given a feasible solution $(\Bar{P},L,G,\alpha_1,\Bar{\alpha}_2,\Bar{\alpha}_3,\Bar{\alpha}_4,\{\theta_i\}_{i=0,1,\ldots,T-1})$ to Problem~\eqref{stability with noise theorem eq}, setting $\Tilde{P} = \Bar{P}^{-1}$ and fixing this $\alpha_1$ in Problem~\eqref{real optimization problem of problem 2} guarantees feasibility of Problem \eqref{real optimization problem of problem 2}.
    \end{mypro}
    \textit{Proof}: The proof follows analogously to that of Proposition~\ref{relationship between three problems} and is omitted for brevity. $\hfill \square$

In analogy to Proposition~\ref{relationship between three problems}, Proposition~\ref{relationship between two problems} motivates an iterative algorithm based on Problem \eqref{real optimization problem of problem 2}, initialized with $\Tilde{P}$ obtained by solving Problem \eqref{stability with noise theorem eq}. The complete procedure is summarized in Algorithm~\ref{algorithm 2}. {\color{blue}Overall, in contrast to the noise-free case, the achievable performance is no longer solely determined by communication and sparsity, but is fundamentally limited by their interaction with disturbances.}

\subsection{$H_\infty$ Performance Analysis}
In this section, we address \textit{Problem 3}, which incorporates the performance output into consideration. 
The performance output is given as
\begin{equation}\label{performance output}
    y(k) = C x(k) + D \hat{u}(k),
\end{equation}
where $C$ and $D$ are known and deterministic matrices. Ideally, the performance of the controller $K$ should be characterized by the maximal $H_\infty$ norm over all $(A,B)\in \Sigma$. However, since our approach is based on the matrix S-procedure, which provides only sufficient  conditions, the resulting performance bound is an upper bound on the maximal $H_\infty$ norm for $\forall(A,B)\in \Sigma$. We refer to this upper bound as the $H_\infty$ norm bound. Our objective is to investigate the trade-off between the achievable $\mathcal{H}_\infty$ performance level and the sparsity of the controller $K$, under event-triggered communication characterized by $\sigma_x$ and $\sigma_u$. {\color{blue}From a networked control perspective, communication constraints and controller sparsity directly influence the attenuation of disturbance signals through the control loop, thereby affecting the achievable $\mathcal{H}_\infty$ performance of the closed-loop system.}

The closed-loop system under consideration is given by
\begin{equation}\label{closed-loop system problem 3}
    \begin{split}
        x(k+1) &= (A  + B K) x(k) - B K e_x(k) - B e_u(k) + E d(k), \\
        y(k) &= C x(k) + D K x(k) - D K e_x(k) - D e_u(k),
    \end{split}
\end{equation}
where $x(k)$, $e_x(k)$, and $e_u(k)$ satisfy \eqref{event-trigger QMI 1} and \eqref{event-trigger QMI 2}. We consider the Lyapunov function $V(k) = x^\top(k) P x(k)$ with $P \succ 0$. According to \citep[Theorem 4.6.6]{skelton2013unified}, the $H_\infty$ norm of the system is less than or equal to $\gamma$ if and only if
\begin{equation}\label{hinf condition}
    V(k+1) - V(k) \leq \gamma^2 d^\top(k) d(k) - y^\top(k) y(k)
\end{equation}
holds for all $k$. Since $A$, $B$, $x(k)$, $e_x(k)$, and $e_u(k)$ are unknown, we require that \eqref{hinf condition} holds for $\forall(A,B)\in\Sigma$, $\forall d(k)$ satisfying $\|d(k)\|_2 \leq \epsilon_d$, and $\forall\left[\begin{smallmatrix} x^\top & e_x^\top(k) & e_u^\top(k) \end{smallmatrix}\right]^\top$ satisfying \eqref{event-trigger QMI 1} and \eqref{event-trigger QMI 2}. {\color{blue}The following theorem establishes a sufficient condition under which the closed-loop system achieves an $H_\infty$ performance level obtained from the proposed design.}
\begin{mythr}\label{Hinf theorem}
    Consider system \eqref{basic system} and \eqref{performance output} with $d(k)$ satisfying $\|d(k)\|_2\leq \epsilon_d$ for all $k\ge0$. Given the event-triggered parameters $\sigma_x,\sigma_u$ in \eqref{event-trigger condition}, if there exist matrices $\Bar{P}\in\mathbb{R}^{n_x\times n_x}$, $L\in\mathbb{R}^{n_u\times n_x}$, $G\in\mathbb{R}^{n_u\times n_u}$, and scalars $\gamma$, $\Bar{\alpha}_1$, $\Bar{\alpha}_2$, $\{\theta_i\}_{i=0,1,\ldots,T-1}$ such that the following holds:
    \begin{subequations}\label{hinf theorem eq}
    \begin{align}
    &\begin{aligned}
        &\begin{bmatrix}
            \mathcal{T}(\Bar{P}) &\mathcal{B}_1(\Bar{P},L,-G) &0 & 0&\Xi_1\\
            * & \Delta_1 & \Xi_2 &\mathcal{B}_2(L^\top,\Bar{P})& 0\\
            * &* &I &0 &0\\
            * &* &* & \Delta_2 &0\\
            * & * &* &* &\gamma^2 I
        \end{bmatrix}\succeq0,\\
        & \Delta_1 = \mathrm{blkdiag}(\Bar{P}, \Bar{\Omega}_1, \Bar{\Theta}_2), \Delta_2 = \mathrm{blkdiag}(\frac{\Bar{\alpha}_2}{\sigma_u^2}I, \frac{\Bar{\alpha}_1}{\sigma_x^2}I),\\
        & \Xi_1 = \begin{bmatrix}E^\top &0 &0\end{bmatrix}^\top, \Xi_2 = \begin{bmatrix}C\Bar{P}+DL &-DL &-DG\end{bmatrix}^\top,
        \end{aligned} \\
    & \Bar{P}\in \mathbb{S}^{n_x}_{++}, \quad \Bar{\alpha}_1\ge0, \quad \Bar{\alpha}_2\ge0, \quad \theta_0\ge0, \ldots, \theta_{T-1}\ge0, 
    \end{align}
\end{subequations} 
then the $H_\infty$ norm of the closed-loop system \eqref{closed-loop system problem 3} is less than or equal to $\gamma$ under the controller $K=L\Bar{P}^{-1}$ with Lyapunov matrix $P=\Bar{P}^{-1}$.
\end{mythr}
\textit{Proof}: The proof follows by combining techniques used in the proofs of \cref{stability without noise theorem} and \cref{stability with noise theorem}, and thus is omitted. $\hfill \square$

For brevity, we directly present the optimization problem that motivates the iterative algorithm for solving \textit{Problem~3}:
 \begin{subequations}\label{real optimization problem of problem 3}
    \begin{align}
        &\mathop{\min}_{Y,P,K,G,\gamma,\alpha_1,\Bar{\alpha}_2,\{\theta_i\}_{i=0,1,\ldots,T-1}} \gamma^2+\lambda\mathrm{Tr}(W^\top K^{|\cdot|})\qquad  \text{s.t.} \nonumber\\
    &\begin{aligned}
        &\begin{bmatrix}
            \mathcal{T}(Y) &\mathcal{B}_1(I,K,-G) &0 & 0&\Xi_1\\
            * & \Delta & \Xi_2 &\mathcal{B}_3(K^\top)& 0\\
            * &* &I &0 &0\\
            * &* &* & \frac{\Bar{\alpha}_2}{\sigma_u^2}I &0\\
            * & * &* &* &\gamma^2 I
        \end{bmatrix}\succeq0,\\
        & \Delta = \mathrm{blkdiag}(P-\alpha_1\sigma_x^2 I, \alpha_1 I, \Bar{\Theta}_2),\\
        & \Xi_1 = \begin{bmatrix}E^\top &0 &0\end{bmatrix}^\top, \Xi_2 = \begin{bmatrix}C+DK &-DK &-DG\end{bmatrix}^\top,
        \end{aligned} \\
    &Y \preceq \Tilde{P}^{-1} - \Tilde{P}^{-1}(P-\Tilde{P})\Tilde{P}^{-1}, \\
        & P\in \mathbb{S}^{n_x}_{++}, Y\in \mathbb{S}^{n_x}_{++},\alpha_1\ge0, \Bar{\alpha}_2\ge0,\\
        & \theta_0\ge0, \theta_1\ge0,\ldots, \theta_{T-1}\ge0.
    \end{align}
\end{subequations} 
The derivation follows analogously to the procedures in the previous two sections. The optimization problem is given in Problem \eqref{real optimization problem of problem 3}.
\begin{algorithm}[t]
    \caption{Iterative algorithm to solve \textit{Problem 3}.}
\textbf{Output:} $K^*,\gamma^*$;\\
    Solve Problem \eqref{hinf theorem eq} and initialize $P_0=\Bar{P}^{-1}, W_0=\mathbf{1}_{n_u\times n_x}, \lambda>0,\epsilon_r>0,k=0$;\\
    \Repeat{\text{convergence or the maximum iteration reached\vspace{5pt}}}{
    - Solve Problem \eqref{real optimization problem of problem 3} with $\Tilde{P} = P_k, W=W_k$;\\
    - Obtain the solutions $K_{k+1}, P_{k+1}, \gamma_{k+1}$;\\
    - $W_{k+1} = (K_{k+1}^{|\cdot|}+\epsilon_r \cdot \mathbf{1}_{n_u\times n_x})^{\circ -1}$;\\
    - $k = k + 1$;
    }
\Return $K_k,\gamma_k$.
\label{algorithm 3}
\end{algorithm} 
The following result establishes the relationship between Problem \eqref{real optimization problem of problem 3} and Problem \eqref{hinf theorem eq}.
\begin{mypro}
    Given a feasible solution $(Y,P,K,G,\gamma,  \alpha_1,\Bar{\alpha}_2,\{\theta_i\}_{i=0,1,\ldots,T-1})$ to Problem~\eqref{real optimization problem of problem 3}, setting $\Tilde{P} = P$ ensures that Problem \eqref{real optimization problem of problem 3} remains feasible. Furthermore, given a feasible solution $(\Bar{P},L,G,\gamma,\Bar{\alpha}_1,\Bar{\alpha}_2,\{\theta_i\}_{i=0,1,\ldots,T-1})$ to Problem~\eqref{hinf theorem eq}, setting $\Tilde{P} = \Bar{P}^{-1}$ guarantees feasibility of Problem~\eqref{real optimization problem of problem 3}.
\end{mypro}
\textit{Proof}: The proof is analogous to that of \cref{relationship between three problems} and is omitted for brevity. $\hfill \square$

The complete algorithm is summarized in \cref{algorithm 3}. {\color{blue}
Overall, the $\mathcal{H}_\infty$ formulation complements the UUB analysis by shifting the focus from state boundedness to disturbance rejection capability. It reveals how communication constraints and controller sparsity jointly affect the achievable $H_\infty$ performance level in data-driven NCSs.
}
{\color{blue}
\begin{myrem}
    The computational complexity of Algorithms \ref{algorithm 1}-\ref{algorithm 3} is mainly determined by solving a sequence of semidefinite programs (SDPs) in an iterative manner. At each iteration, an SDP with decision variables and matrix inequality constraints whose sizes scale with the system dimension is solved. Following the interior-point method framework for SDP \cite{vandenberghe1996semidefinite}, the resulting computational complexity can be characterized as $\mathcal{O}(N_{max}\sqrt{n_x+n_u}(n_x^2+n_u^2+n_xn_u+T)^3)$, where $N_{max}$ denotes the maximum number of iterations. The corresponding memory complexity scales as $\mathcal{O}((n_x^2+n_u^2+n_xn_u+T)^2)$. Therefore, both the computational and memory complexities grow polynomially with respect to the system dimensions and the data length. In the implementation, although we set $N_{max}=150$, the algorithms typically converge within tens to around one hundred iterations, indicating that the overall computational cost remains tractable for systems of moderate size. It is worth emphasizing that the above computational cost corresponds to the offline design stage. In practical NCSs, controller synthesis can typically be performed offline using relatively powerful computational resources, where higher computational complexity can be tolerated. In contrast, the online implementation is executed on resource-constrained platforms and must meet real-time requirements. Once the controller and triggering mechanism are obtained, the online implementation of the proposed framework only involves evaluating simple triggering conditions and computing control inputs via sparse controller gains, which is computationally lightweight and requires limited memory. Therefore, the proposed approach effectively shifts the computational burden to the offline stage, while significantly reducing both communication and computation costs during online operation.
\end{myrem}
\begin{myrem}
    It should be noted that, even if the uncertainty set $\Sigma$ is bounded, the feasibility of Problems \eqref{stability without noise theorem eq}, \eqref{stability with noise theorem eq}, and \eqref{hinf theorem eq} is not guaranteed in general. This is because feasibility depends not only on the size of $\Sigma$, which is determined by the noise level and data informativeness, but also on the conservativeness of the proposed conditions. Specifically, the results provide verifiable sufficient conditions for controller synthesis, which may be conservative due to two aspects: (i) the derivation steps, including the (lossy) matrix S-lemma, bounding and decoupling techniques, and the use of a common Lyapunov function, and (ii) the requirement to ensure robustness over the data-driven uncertainty set. As a result, infeasibility may stem either from insufficiently informative data or from relaxation-induced conservativeness. When the noise level decreases, the data become more informative and the uncertainty set shrinks, which improves feasibility and reduces conservativeness.
\end{myrem}
\begin{myrem}
    Algorithms \ref{algorithm 1}–\ref{algorithm 3} are iterative reweighting-type procedures and are heuristic in nature. Similar to reweighted $\ell_1$ minimization methods widely studied in the literature (e.g., \cite{candes2008enhancing}), establishing rigorous global convergence guarantees is generally challenging. Nevertheless, such algorithms have been extensively used in practice and are observed to exhibit reliable convergence behavior. In our implementation, each algorithm is terminated when either the maximum number of iterations ($150$) is reached, or the relative change in the objective value between two consecutive iterations falls below a prescribed tolerance set to $10^{-4}$. In all numerical experiments, the algorithms typically converge within a few tens of iterations, well before reaching the maximum iteration limit.
\end{myrem}
}
\section{Simulations}\label{Simulations}
In this section, we present three simulation examples to demonstrate the effectiveness and scalability of the proposed methods in addressing \textit{Problem 1}, \textit{Problem 2}, and \textit{Problem 3}, respectively, together with comparisons to existing literature. All examples are discretized using the zero-order hold (ZOH) method with sampling period $T_s = 0.1$s and {\color{blue} implemented within the NCS framework defined in \cref{Event-trigger control}, where S2C and C2A communication channels are explicitly present, and the sparsity pattern of $K$ characterizes the underlying controller information topology.} For each example, we assume access to a sequence of noisy measurements of system states and control inputs with data length $T=100$. The state measurement noise $\delta_x(k)$ and input measurement noise $\delta_u(k)$ are both assumed to be bounded, with $\epsilon_x = \epsilon_u = 0.001$. The bound for the process noise $d(k)$ will be specified individually for each simulation example. {\color{blue}Throughout this section, we use $r_x$ and $r_u$ to represent the transmission rates over the S2C and C2A channels, respectively, and are computed as the ratio between the number of triggering events and the simulation length}.
All simulations are performed on a desktop equipped with an Intel Core i7-12700H CPU and 16 GB RAM, using MATLAB R2024b and Mosek.

\subsection{Example 1: Reactor Model}\label{Example 1}
In this example, we address \textit{Problem 1}. We consider the REA1 example from the standard benchmark collection COMPleib \cite{leibfritz2003description}, which models the dynamics of a chemical reactor. The corresponding discrete-time state-space model is given by \eqref{basic system} with $d(k)=0$, where
\begin{equation*}
        \small{A_*} = \begin{bmatrix}\begin{smallmatrix}
            1.1782 & 0.0015 & 0.5116 & -0.4033 \\
            -0.0515 & 0.6619 & -0.0110 & 0.0613 \\
            0.0762 & 0.3351 & 0.5606 & 0.3824 \\
            -0.0006 & 0.3353 & 0.0893 & 0.8494
        \end{smallmatrix}\end{bmatrix}, 
        \small{B_*} = \begin{bmatrix}\begin{smallmatrix}
            0.0045 & -0.0876 \\
            0.4672 & 0.0012 \\
            0.2132 & -0.2353 \\
            0.2131 & -0.0161
        \end{smallmatrix}\end{bmatrix}.
\end{equation*}
Since $d(k)=0$, we collect data in the absence of the process noise and use $\epsilon_d=0$ to describe the set of $(A,B)$ consistent with data. We set $\beta = 0.999$ and choose the event-triggered parameters as $\sigma_x = 0.1581$ and $\sigma_u = 0.01$. By executing \cref{algorithm 1}, we obtain the following sparse controller:
\begin{equation*}
    K^* = \begin{bmatrix}\begin{smallmatrix}
        0 & 0 & 0 & 0 \\
        1.3849 & 0 & 0 & 0
    \end{smallmatrix}\end{bmatrix}.
\end{equation*}
This result indicates that only the second actuator is active, and its computation depends solely on the first state variable, as the first row of $K^*$ is identically zero. We simulate the closed-loop system with the initial condition $x(0) = \left[\begin{smallmatrix} -1 & -2 & 2 & -2 \end{smallmatrix}\right]^\top$. The trajectories of the system states and actuator control signals are depicted in \cref{problem1 plot}.

As shown in \cref{problem1:x}, the system states converge to the origin within approximately $30$ time steps ($3$ seconds). \cref{problem1:hatu} demonstrates that the first actuator remains inactive, while the second actuator is updated intermittently, reflecting the effect of the event-triggered mechanism. The transmission rates over the S2C and C2A channels are $68.33\%$ and $66.67\%$, respectively. 
{\color{blue}
To further illustrate the impact of the event-triggered parameters, multiple case studies with different $(\sigma_x, \sigma_u)$ are reported in \cref{Effect of Event-Triggered Parameters for example 1}.

Comparing the first two rows shows that increasing $\sigma_u$ significantly reduces the C2A transmission rate. In addition, comparing the first two rows with the next two rows shows that increasing $\sigma_x$ effectively reduces both the S2C and C2A transmission rates. It is worth noting that excessively low transmission rates may lead to less sparse stabilizing controllers. This is because both ETC and SC may impair system stability, revealing an inherent trade-off between them. According to the unified cost model, the third and fourth rows are preferable to the first two rows, while the specific choice depends on the relative emphasis on ETC and SC (i.e., the weighting coefficients in the unified cost model). Finally, the fifth row shows that excessively large event-triggered parameters may lead to infeasibility, which is consistent with the previous discussion.
}
\begin{figure}[t]
    \centering
    \hfill 
    \begin{subfigure}[b]{0.49\columnwidth}
        \includegraphics[width=\linewidth]{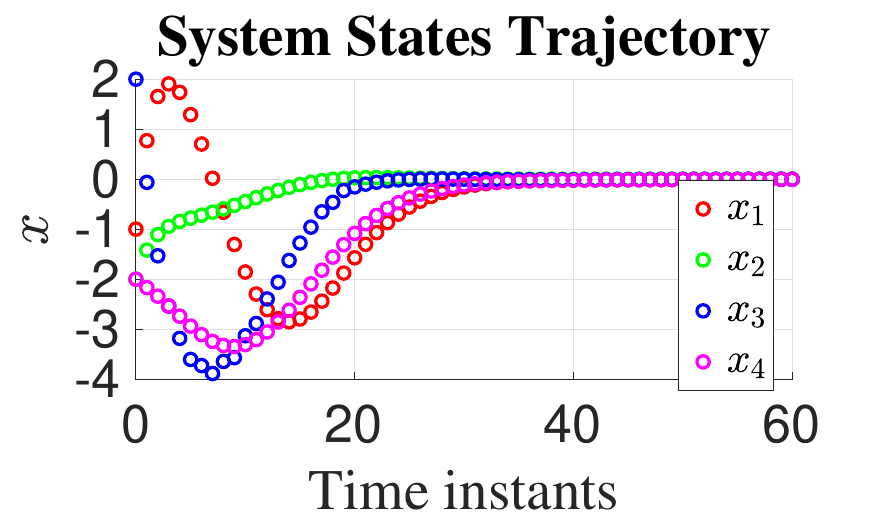}
        \caption{State trajectory $x$.}
        \label{problem1:x}
    \end{subfigure}
    \hfill
    \begin{subfigure}[b]{0.49\columnwidth}
        \includegraphics[width=\linewidth]{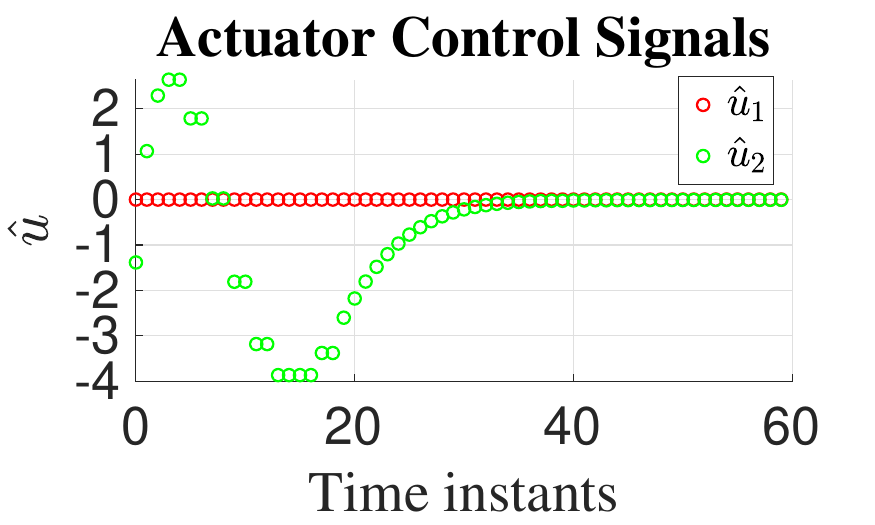}
        \caption{Control input $\hat{u}$.}
        \label{problem1:hatu}
    \end{subfigure}
    \caption{{\color{blue}State and control input trajectories of the closed-loop system without process noise, showing state convergence and intermittent control updates.}}
    \label{problem1 plot}
\end{figure}
\begin{table}[!t]
\centering
\caption{Effect of Event-Triggered Parameters on Feasibility, Communication Rate, and Controller Sparsity.}
\label{Effect of Event-Triggered Parameters for example 1}
{\color{blue}
\begin{tabular}{c c c c c c}
\toprule
$\sigma_x$ &$ \sigma_u$ & $r_x$ & $r_u$ & $\|K\|_0$ & Feasibility\\
\midrule
$0.0316$& $0.0316$ & $100\%$ &$91.67\%$ & $1$ & Feasible  \\
$0.0316$& $0.4472$ & $100\%$ &$31.67\%$ & $3$ &Feasible \\
$0.1581$& $0.0100$ & $68.33\%$ &$66.67\%$ & $1$ &Feasible \\
$0.1690$& $0.0447$ & $55.00\%$ &$46.67\%$ & $2$ &Feasible \\
$0.1826$& $0.1000$ & $\times$ &$\times$ & $\times$ &Infeasible \\
\bottomrule
\end{tabular}
}
\end{table}
\subsection{Example 2: Mass-Spring System}
In this example, we address \textit{Problem 2}. We consider the mass-spring system from \cite{lin2013design} with $N=2$ masses. The corresponding state-space model is given by \eqref{basic system} with
\begin{equation*}
    \begin{split}
    \small{A_* =} \begin{bmatrix}\begin{smallmatrix}
        0.9900	&0.0050	&0.0997	&0.0002\\
        0.0050	&0.9900	&0.0002	&0.0997\\
        -0.1992	&0.0993	&0.9900	&0.0050\\
        0.0993	&-0.1992	&0.0050	&0.9900
        \end{smallmatrix}\end{bmatrix}, \small{B_* = E = }\begin{bmatrix}\begin{smallmatrix}
        0.0050	&0.0000\\
        0.0000	&0.0050\\
        0.0997	&0.0002\\
        0.0002	&0.0997
        \end{smallmatrix}\end{bmatrix}.
    \end{split}
\end{equation*}
The process noise $d(k)$ is assumed to be bounded with $\epsilon_d=0.1$. We set $\alpha_1=0.05,\beta=0.999$ and choose the event-triggered parameters as $\sigma_x=0.1414$ and $\sigma_u=0.3162$. We first consider the dense controller case and subsequently investigate the trade-off between controller sparsity and performance. By executing \cref{algorithm 2} with $\lambda=0$ in Problem~\eqref{real optimization problem of problem 2}, we obtain
\begin{equation*}
        K^*=\begin{bmatrix}\begin{smallmatrix}
            -4.9938	&-1.0802	&-7.8510	&0.0116\\
            -0.5705	&-4.8836	&-0.1037	&-8.1137
        \end{smallmatrix}\end{bmatrix},\mathrm{Tr}({P^*}^{-1})=0.0150.
\end{equation*}
We simulate the closed-loop system with the initial condition $x(0)=\left[\begin{smallmatrix}1 &-2 &2 &-1\end{smallmatrix}\right]^\top$. The trajectories of the system states and actuator control signals are shown in \cref{problem2 plot}.

As illustrated in \cref{problem2 plot}, both the system states and actuator control signals ultimately converge to a neighborhood of the origin, which is consistent with the UUB property established in \cref{stability with noise theorem}. This region can be characterized by the set $\mathcal{S}$ defined in \eqref{bounded set} with the Lyapunov matrix $P^*$. Furthermore, \cref{problem2:hatu} demonstrates that there are several time steps when the actuators are not updated, reflecting the effect of the event-triggered mechanism. The transmission rates over the S2C and C2A channels are $68.33\%$ and $60.00\%$, respectively.

\begin{figure}[t]
    \centering
    \hfill 
    \begin{subfigure}[b]{0.49\columnwidth}
        \includegraphics[width=\linewidth]{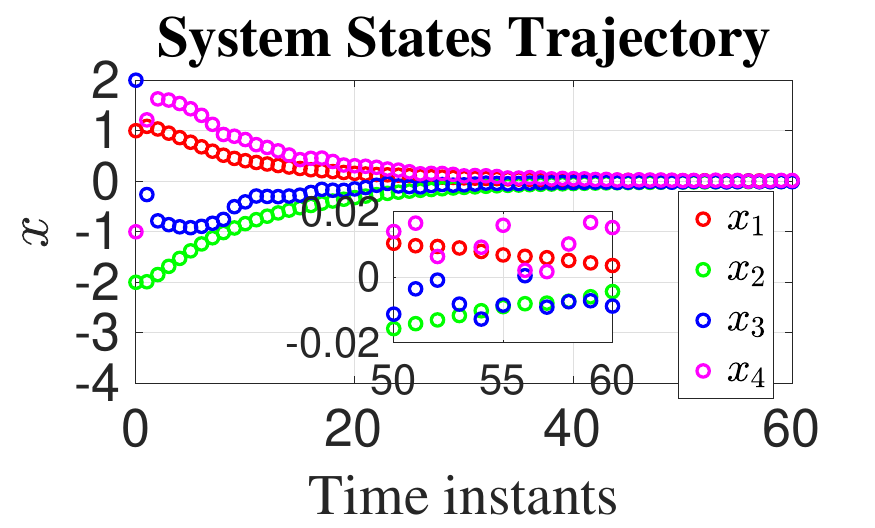}
        \caption{State trajectory $x$.}
        \label{problem2:x}
    \end{subfigure}
    \hfill
    \begin{subfigure}[b]{0.49\columnwidth}
        \includegraphics[width=\linewidth]{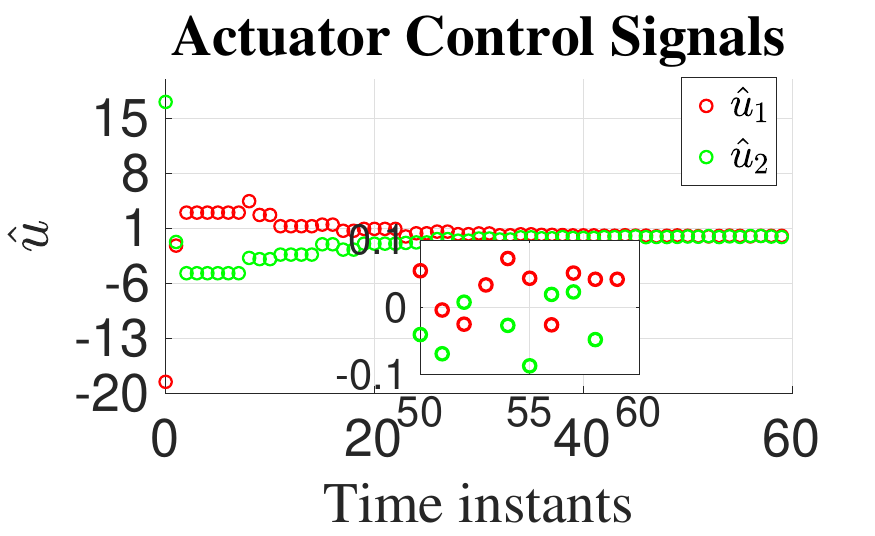}
        \caption{Control input $\hat{u}$.}
        \label{problem2:hatu}
    \end{subfigure}
    \caption{{\color{blue}State and control input trajectories of the closed-loop system with process noise, showing convergence of the state to a neighborhood of the origin and intermittent control updates.}}
    \label{problem2 plot}
\end{figure}
\begin{table}[!t]
\centering
\caption{Effect of $\lambda$ on the size of UUB Set ($\mathrm{Tr}(P^{-1})$) and Controller Sparsity in \cref{algorithm 2}.}
\label{Simulation results for different lambda in algorithm 2}
\begin{tabular}{c c c c c}
\toprule
$\lambda$ & $\mathrm{Tr}(P^{-1})$ & $||K||_0$ & $r_x$ & $r_u$ \\
\midrule
$0$ & $0.0150$ & $8$ & $68.33\%$ &$60.00\%$ \\
$1e-4$ & $0.0157$ & $5$ &$68.33\%$ &$56.67\%$ \\
$1e-2$ & $0.0699$ & $4$ &$73.33\%$ &$56.67\%$ \\
\bottomrule
\end{tabular}
\end{table}
\begin{table}[!t]
\centering
\caption{Effect of Event-Triggered Parameters and the Bound of Process Noise on the Size of UUB Set ($\mathrm{Tr}(P^{-1})$) and Communication Rate (with $\lambda = 0$).}
\label{Effect of Event-Triggered Parameters for example 2}
{\color{blue}
\begin{tabular}{c c c c c c}
\toprule
$\sigma_x$ &$ \sigma_u$& $\epsilon_d$ & $r_x$ & $r_u$ & $\mathrm{Tr}(P^{-1})$\\
\midrule
$0.0316$ & $0.1000$ & $0.1$ & $100\%$ & $73.00\%$  & $0.0074$ \\
$0.0316$ & $0.1000$ & $0.06$ & $100\%$ & $65.00\%$  & $0.0027$ \\
$0.1414$ & $0.3162$ & $0.1$ & $68.33\%$ & $60.00\%$ & $0.0150$\\
$0.1414$ & $0.3162$ & $0.06$ & $66.67\%$& $51.67\%$  & $0.0047$\\
$0.3162$ & $0.4472$ & $0.06$ & $\times$ & $\times$ & $\times$\\
\bottomrule
\end{tabular}
}
\end{table}
Next, we examine the trade-off between the sparsity of the controller and the size of $\mathcal{S}$. We run \cref{algorithm 2} with different values of $\lambda$ in Problem \eqref{real optimization problem of problem 2}. As shown in \cref{Simulation results for different lambda in algorithm 2}, increasing $\lambda$ yields sparser controllers, but the size of $\mathcal{S}$ (measured by $\mathrm{Tr}(P^{-1})$) increases. The transmission rates remain largely unaffected, since they are primarily determined by the event-triggered parameters $\sigma_x$ and $\sigma_u$. This trade-off is expected, since a larger $\lambda$ imposes a higher penalty on the non-sparsity of $K$, potentially at the expense of closed-loop performance.

{\color{blue}
To further illustrate the impact of the event-triggered parameters and the process noise bound, multiple case studies with different $(\sigma_x, \sigma_u, \epsilon_d)$ are reported in \cref{Effect of Event-Triggered Parameters for example 2}, where $\lambda=0$ and sparsity is not enforced. It can be observed that the effect of the event-triggered parameters on the transmission rates is consistent with the noise-free case, i.e., larger $(\sigma_x, \sigma_u)$ lead to lower communication rates. Comparing the first two rows, as well as the third and fourth rows, shows that a larger process noise bound $\epsilon_d$ results in a larger UUB set. Moreover, comparing the first and third rows, as well as the second and fourth rows, indicates that lower transmission rates also lead to a larger UUB set. This is because process noise enlarges the UUB set, while reduced communication under ETC limits feedback information. Finally, the last row shows that excessively large triggering parameters render the problem infeasible, which is consistent with the noise-free case.
}
\begin{figure}[t]
    \centering
    \hfill 
    \begin{subfigure}[b]{0.49\columnwidth}
        \includegraphics[width=\linewidth]{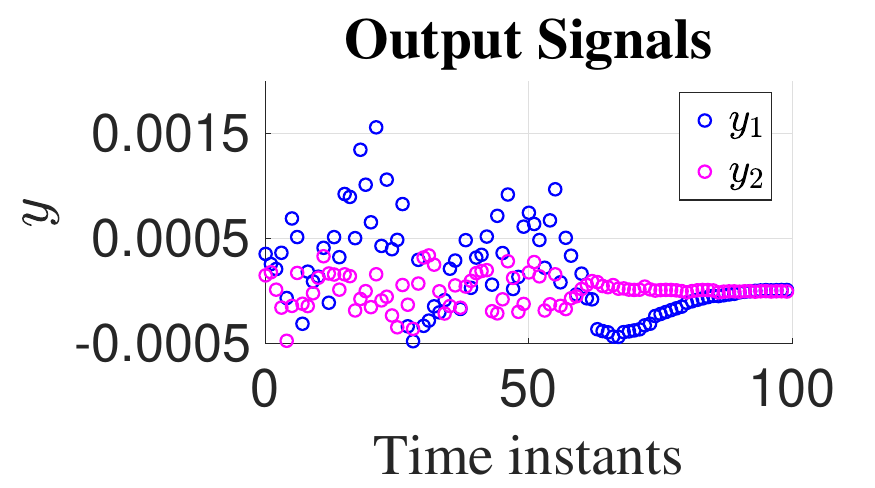}
        \caption{Output trajectory $y$.}
        \label{problem3:y}
    \end{subfigure}
    \hfill
    \begin{subfigure}[b]{0.49\columnwidth}
        \includegraphics[width=\linewidth]{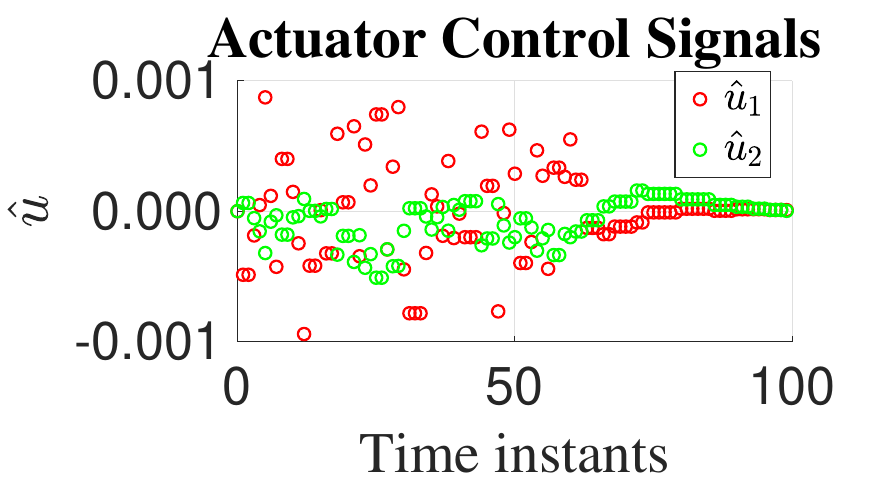}
        \caption{Control input $\hat{u}$.}
        \label{problem3:hatu}
    \end{subfigure}
    \caption{{\color{blue}Output and control input trajectories of the closed-loop system under the $H_\infty$ setting, showing small output energy and intermittent control updates.}}
    \label{problem3 plot}
\end{figure}
\begin{table}[!t]
\centering
\caption{Effect of $\lambda$ on the $H_\infty$ norm bound $\gamma$ and Controller Sparsity in \cref{algorithm 3}.}
\label{Simulation results for different lambda in algorithm 3}
\begin{tabular}{c c c c c}
\toprule
$\lambda$ & $\gamma$ & $||K||_0$ & $r_x$ & $r_u$ \\
\midrule
$0$ & $0.1823$ & $8$ & $88.00\%$ &$59.00\%$ \\
$1e-4$ & $0.1963$ & $6$ &$87.00\%$ &$56.00\%$ \\
$1e-2$ & $0.2156$ & $5$ &$89.00\%$ &$54.00\%$ \\
$1e-1$ & $0.3779$ & $4$ &$89.00\%$ &$57.00\%$ \\
\bottomrule
\end{tabular}
\end{table}
\begin{table}[!t]
\centering
\caption{Effect of Event-Triggered Parameters on the $H_\infty$ norm bound and Communication Rate (with $\lambda = 0$).}
\label{Effect of Event-Triggered Parameters for example 3}
{\color{blue}
\begin{tabular}{c c c c c}
\toprule
$\sigma_x$ &$ \sigma_u$ & $r_x$ & $r_u$ & $\gamma$\\
\midrule
$0.1000$ & $0.1000$  & $97.00\%$ & $94.00\%$  & $0.1058$ \\
$0.1414$ & $0.4472$  & $88.00\%$ & $59.00\%$ & $0.1823$\\
$0.3162$ & $0.1000$  & $\times$ & $\times$ & $\times$\\
\bottomrule
\end{tabular}
}
\end{table}
\subsection{Example 3: Helicopter Model}
In this example, we address \textit{Problem 3}. We consider the example HE1 from COMPleib \cite{leibfritz2003description}, which describes the dynamics of a helicopter. The corresponding discrete-time state-space model is given by \eqref{basic system} and \eqref{performance output} with
\begin{equation*}
    \begin{split}
        &\small{A_*}= \begin{bmatrix}\begin{smallmatrix}
            0.9964	&0.0026	&-0.0004	&-0.0460\\
0.0045	&0.9037	&-0.0188	&-0.3834\\
0.0098	&0.0339	&0.9383	&0.1302\\
0.0005	&0.0017	&0.0968	&1.0067
            \end{smallmatrix}\end{bmatrix}, 
        \small{B_*} =  \begin{bmatrix}\begin{smallmatrix}
            0.0445	&0.0167\\
0.3407	&-0.7249\\
-0.5278	&0.4214\\
-0.0268	&0.0215
            \end{smallmatrix}\end{bmatrix},\\ 
            &\scriptstyle{E} = \begin{bmatrix}\begin{smallmatrix}
                0.0047	&0.0000\\
0.0048	&0.0009\\
0.0042	&0.0001\\
-0.0020	&0.0000
            \end{smallmatrix}\end{bmatrix}, 
        \scriptstyle{C} = \begin{bmatrix}\begin{smallmatrix}
            1.4142	&0\\
0	&0.7071\\
0 &0\\ 
0 &0
            \end{smallmatrix}\end{bmatrix}^\top, \scriptstyle{D}=\begin{bmatrix}\begin{smallmatrix}
            0.7071	&0\\
0	&0.7071
            \end{smallmatrix}\end{bmatrix}.
    \end{split}
\end{equation*}
We set the event-triggered parameters to $\sigma_x=0.1414$ and $\sigma_u=0.4472$, and the maximum iteration number to $20$. We first use the dense controller case to illustrate system behavior. By running \cref{algorithm 3} with $\lambda=0$ in Problem \eqref{real optimization problem of problem 3}, we obtain

\begin{equation*}
K^*  = \begin{bmatrix}\begin{smallmatrix}
    -1.6058	&-0.3153	&0.6056	&1.4077\\
-0.7829	&0.7476	&0.3062	&-0.3214
    \end{smallmatrix}\end{bmatrix} ,\gamma^* = 0.1823.\\
\end{equation*}
We then simulate the closed-loop system with the initial condition $x(0)=\left[\begin{smallmatrix}0 &0 &0 &0\end{smallmatrix}\right]^\top$. To construct a bounded process noise sequence such that $d \in l_2[0,+\infty)$, we assume $d(k)$ satisfies $\|d(k)\|_2 \leq 0.1$ (i.e., $\epsilon_d=0.1$) for the first $60$ time steps, and $d(k)=0$ thereafter. The output signals and actuator control signals are depicted in \cref{problem3 plot}.

As shown in \cref{problem3:y}, the output signal $y(k)$ eventually converges to the origin and remains relatively small compared to $d(k)$ with $\|d(k)\|_2 \leq 0.1$, demonstrating the robustness of the closed-loop system against external disturbances. Furthermore, \cref{problem3:hatu} illustrates that there are several time steps during which the actuators are not updated, reflecting the effect of the event-triggered mechanism. The transmission rates over the S2C and C2A channels are $88.00\%$ and $59.00\%$, respectively.

Next, we investigate the trade-off between controller sparsity and $H_\infty$ performance. We execute \cref{algorithm 3} with different values of $\lambda$ in Problem \eqref{real optimization problem of problem 3}. As observed in \cref{Simulation results for different lambda in algorithm 3}, increasing $\lambda$ yields sparser controllers, but the $H_\infty$ norm bound $\gamma$ increases. This is expected, as a larger $\lambda$ imposes a higher penalty on the non-sparsity of $K$, potentially at the expense of closed-loop performance. Additionally, transmission rates vary slightly, consistent with prior discussions.

{\color{blue}
Similar to the previous examples, multiple case studies with different $(\sigma_x, \sigma_u)$ are reported in \cref{Effect of Event-Triggered Parameters for example 3}. The effect of the event-triggered parameters on the transmission rates, as well as the infeasibility issue for large triggering parameters, is consistent with the previous examples. Specifically, lower transmission rates lead to degraded $H_\infty$ performance, which is consistent with the observation in \textit{Example~2}.
}
{\color{blue}
\subsection{Scaling Analysis}
In this subsection, we investigate the computational scalability of Algorithms \ref{algorithm 1}--\ref{algorithm 3} with respect to the system dimension. To this end, we construct higher-dimensional systems by stacking multiple independent subsystems in the following way 
\begin{equation*}
\begin{split}
    &{A_n}_* = \mathrm{blkdiag}(\overbrace{A,\ldots,A}^{n}), {B_n}_* = \mathrm{blkdiag}(\overbrace{B,\ldots,B}^{n}),\\
    &C_n = \mathrm{blkdiag}(\overbrace{C,\ldots,C}^{n}),D_n = \mathrm{blkdiag}(\overbrace{D,\ldots,D}^{n}),\\
    &E_n = [\overbrace{E^\top|\ldots |E^\top}^{n}]^\top.
\end{split}
\end{equation*}
To ensure a fair comparison across different system dimensions, each algorithm is executed with a fixed number of iterations (e.g., $50$), and the total computational time is recorded accordingly. This setting eliminates the influence of different convergence behaviors and allows us to isolate the effect of the problem size on the computational cost. The results are summarized in \cref{scaling analysis}. Here, $n$ denotes the number of stacked subsystems, and the resulting system dimensions scale as $n_x^{\mathrm{total}} = n n_x$ and $n_u^{\mathrm{total}} = n n_u$. It can be observed that the computational time increases with the system dimension for all three algorithms, which is consistent with the growth in the size of the underlying SDPs. Although the proposed framework can be applied to systems of arbitrary dimensions, we report results for a set of representative system sizes that clearly illustrate this offline scaling behavior. 
\begin{table}[t]
\centering
\caption{Scaling analysis of the proposed algorithms.}
{\color{blue}
\begin{tabular}{cccccc}
\toprule
$n$& $n_x^{total}$ & $n_u^{total}$& Algorithm 1 & Algorithm 2& Algorithm 3\\
\midrule
$1$ & $4$ &$2$ &$17.1859$s & $17.7669$s & $16.8633$s\\

$2$ & $8$ &$4$ &$38.4342$s & $35.9761$s & $24.8666$s \\

$3$ & $12$ &$6$ &$65.1903$s & $59.2414$s & $41.8802$s\\
\bottomrule
\end{tabular}
}
\label{scaling analysis}
\end{table}
\subsection{Comparisons with Existing Literature}
In this subsection, we compare the proposed methods with existing data-driven approaches that focus solely on either SC or ETC design. Specifically, we compare our methods with the following baselines: (i) a data-driven SC design method that does not consider ETC \cite{eising2022informativity}, and (ii) a data-driven ETC design method that does not consider SC \cite{feng2025data}. 
\subsubsection{Comparison with \cite{eising2022informativity}}
Eising and Cort{\'e}s \cite{eising2022informativity} proposed an iterative algorithm to compute a robust stabilizer with maximal sparsity. However, they only assume $A$ to be unknown and require the input matrix $B$ to be known. Besides, they do not consider the effect of measurement noise and only assume the process noise sequence is energy-bounded. Therefore, our data-driven framework is more general and can handle a wider range of practical scenarios. To ensure a fair comparison, we apply the method in \cite{eising2022informativity} to our data-driven framework and compare its operational cost with that of our method. We consider the same setting as the reactor model in \cref{Example 1} but choose a different data length $T=50$. By applying \cref{algorithm 1} and simulating the closed-loop system, we obtain a sparse controller with $||K||_0=1$ and transmission rates $r_x=68.33\%, r_u=66.67\%$. In contrast, the algorithm in \cite{eising2022informativity} returns a sparse controller with $||K||_0=5$. Since \cite{eising2022informativity} does not consider ETC, $r_x=r_u=100.00\%$. Both control strategies can stabilize the system. Nevertheless, the operational cost of our method is significantly lower than that of \cite{eising2022informativity} according to the unified cost model introduced in \cref{unified cost model}.
\subsubsection{Comparison with \cite{feng2025data}}
Feng et al. \cite{feng2025data} proposed to co-design a controller (not necessarily sparse) and $\sigma_u$ to stabilize the closed-loop system. Their data-driven framework differs from ours in that they consider the event-triggered state measurements. To ensure a fair comparison, we apply the method in \cite{feng2025data} to our data-driven framework. We consider exactly the same setting as the reactor model in \cref{Example 1}. As recorded in \cref{Example 1}, \cref{algorithm 1} returns a sparse controller with $||K||_0=1$ and transmission rates $r_x \approx 68\%, r_u \approx 67\%$. In contrast, the method in \cite{feng2025data} leads to $r_u \approx 42\%$. Since  \cite{feng2025data} does not consider SC and only focuses on ETC in the C2A channel, $||K||_0=8$ and $r_x=100\%$. By substituting the known parameters into the unified cost model in \cref{unified cost model}, the operational cost of our method can be expressed as 
\begin{equation*}
    J_{ours} = 2.72\alpha_x + 1.34\alpha_u + \alpha_m + \alpha_f.
\end{equation*}
The operational cost of the method in \cite{feng2025data} is expressed as
\begin{equation*}
    J_{base} = 4\alpha_x + 0.84\alpha_u + 8(\alpha_m + \alpha_f) + 6\alpha_a.
\end{equation*}
The proposed method achieves a lower operational cost whenever $J_{ours} < J_{base}$, which is equivalent to
\begin{equation*}
    1.28\alpha_x + 7(\alpha_m+\alpha_f) + 6\alpha_a > 0.5\alpha_u.
\end{equation*}
This condition is typically satisfied in practical implementations when the S2C communication cost and the intra-controller computation and dataflow costs are non-negligible.
}

\section{Conclusion}\label{Conclusions}
This paper investigated the data-driven co-design of SC and ETC for unknown systems. A novel data-driven framework was proposed to accommodate bounded state and input measurement noise as well as bounded process noise. Within this framework, three fundamental problems were addressed: (i) stability analysis in the absence of process noise; (ii) UUB analysis with bounded process noise; and (iii) $H_\infty$ performance analysis. For each problem, sufficient conditions for the existence of a feasible controller were derived, and iterative algorithms were developed to solve the associated nonconvex optimization problems. The effectiveness of the proposed methods was verified through numerical examples.

\appendix
\subsection{Proof of \cref{lemma:joint bound of process noise and measurement noise}}\label{proof of lemma:joint bound of process noise and measurement noise}
For $\forall{z} = \begin{bmatrix}\begin{smallmatrix}
        z_1^\top & z_2^\top & z_3^\top & z_4^\top
        \end{smallmatrix}\end{bmatrix}^\top$, where $z_1 \in \mathbb{R}^{n_d}, z_2 \in \mathbb{R}^{n_x}, z_3 \in \mathbb{R}^{n_x}$, and $z_4 \in \mathbb{R}^{n_u}$, we have 
\begin{equation*}
    \begin{aligned}
    &z^\top \begin{bmatrix}\begin{smallmatrix}
            d\\
            \delta_x\\
            \delta_x\\
            \delta_u
         \end{smallmatrix}\end{bmatrix}
        \begin{bmatrix}\begin{smallmatrix}
            d\\
            \delta_x\\
            \delta_x\\
            \delta_u
         \end{smallmatrix}\end{bmatrix}^\top z  =(z_1^\top d+z_2^\top \delta_x + z_3^\top \delta_x + z_4^\top \delta_u)^2,
    \end{aligned}
\end{equation*}
where we omit the time index $k$ for brevity. By the Cauchy-Schwarz inequality, we have
\begin{equation*}
    \begin{aligned}
    &(z_1^\top d+z_2^\top \delta_x + z_3^\top \delta_x + z_4^\top \delta_u)^2\\
    & \leq 4((z_1^\top d)^2 + (z_2^\top \delta_x)^2 + (z_3^\top \delta_x)^2 + (z_4^\top \delta_u)^2)\\
    & \leq 4(\|z_1\|_2^2\|d\|_2^2 + \|z_2\|_2^2\|\delta_x\|_2^2 + \|z_3\|_2^2\|\delta_x\|_2^2 + \|z_4\|_2^2\|\delta_u\|_2^2)\\
    & \leq 4(\epsilon_d^2\|z_1\|_2^2 + \epsilon_x^2\|z_2\|_2^2 + \epsilon_x^2\|z_3\|_2^2 + \epsilon_u^2\|z_4\|_2^2)\\
    & = z^\top \mathrm{blkdiag}(4\epsilon_d^2 I, 4\epsilon_x^2 I, 4\epsilon_x^2 I, 4\epsilon_u^2 I) z.    
    \end{aligned}
\end{equation*}
According to the definition of the semidefinite matrix, the proof of \cref{lemma:joint bound of process noise and measurement noise} is completed.$\hfill \square$
{\color{blue}
\subsection{Proof of \cref{boundedness of uncertainty set}}
\label{proof of lemma:boundedness of uncertainty set}
This proof follows by proving that an alternative uncertainty set $\Gamma$ satisfying $\Gamma \supseteq \Sigma$ is bounded. Instead of first constructing the uncertainty set of $(A,B)$ consistent with $(x_m(k),u_m(k),x_m(k+1))$ and then making the intersection for $T$ such samples, we directly construct the uncertainty set of $(A,B)$ consistent with the entire data sequence $(X_m^+,X_m^-,U_m^-)$. We first modify \eqref{compact way} to the following form
\begin{equation*}
    \begin{bmatrix}
        I\\
        A^\top\\
        B^\top
    \end{bmatrix}^\top
    \begin{bmatrix}
        X_m^+\\
        -X_m^-\\
        -U_m^-
    \end{bmatrix}
    =
    \begin{bmatrix}
        I\\
        A^\top\\
        B^\top
    \end{bmatrix}^\top
    G
    \begin{bmatrix}
        D\\
        \Delta_x^-\\
        \Delta_x^+\\
        \Delta_u^-
    \end{bmatrix},
\end{equation*}
where we denote 
\begin{equation*}
\begin{aligned}
    D  &\triangleq \begin{bmatrix}d(0) &d(1) &\dots &d(T-1)\end{bmatrix} \in \mathbb{R}^{n_d \times T},\\
    \Delta_x^- &\triangleq \begin{bmatrix}\delta_x(0) &\delta_x(1) &\dots &\delta_x(T-1)\end{bmatrix} \in \mathbb{R}^{n_x \times T},\\
    \Delta_x^+ &\triangleq \begin{bmatrix}\delta_x(1) &\delta_x(2) &\dots &\delta_x(T)\end{bmatrix} \in \mathbb{R}^{n_x \times T},\\
    \Delta_u^- &\triangleq \begin{bmatrix}\delta_u(0) &\delta_u(1) &\dots &\delta_u(T-1)\end{bmatrix} \in \mathbb{R}^{n_u \times T}. 
\end{aligned}
\end{equation*}
We have the following inequality
\begin{equation*}
    \begin{bmatrix}
        D\\
        \Delta_x^-\\
        \Delta_x^+\\
        \Delta_u^-
    \end{bmatrix}\begin{bmatrix}
        D\\
        \Delta_x^-\\
        \Delta_x^+\\
        \Delta_u^-
    \end{bmatrix}^\top = \sum_{k=0}^{T-1} \begin{bmatrix}
        d(k)\\
        \delta_x(k)\\
        \delta_x(k+1)\\
        \delta_u(k)
    \end{bmatrix}\begin{bmatrix}
        d(k)\\
        \delta_x(k)\\
        \delta_x(k+1)\\
        \delta_u(k)
    \end{bmatrix}^\top \leq T\cdot \Pi.
\end{equation*}
We can define an alternative uncertainty set $\Gamma$ as follows
\begin{equation*}
    \Gamma = \left\{ (A,B) \Bigg|     \begin{bmatrix}
        I\\
        A^\top\\
        B^\top
    \end{bmatrix}^\top
    \begin{bmatrix}
        \setlength{\dashlinegap}{0.8pt}
        \begin{array}{c:c}
        \mathbf{C} &\mathbf{B}^\top\\
        \hdashline
        \\[-10pt]
        \mathbf{B} &\mathbf{A}
        \end{array}
        \end{bmatrix}
    \begin{bmatrix}
        I\\
        A^\top\\
        B^\top
    \end{bmatrix}\succeq 0\right\},
\end{equation*}
where 
\begin{equation*}
    \begin{aligned}
    \mathbf{B} &= \begin{bmatrix} X_m^- \\ U_m^- \end{bmatrix} {X_m^+}^\top, \mathbf{C} = T(4\epsilon_d^2E^2+4\epsilon_x^2 I)-X_m^+ {X_m^+}^\top,\\
      \mathbf{A} &= T\cdot \mathrm{blkdiag}(4\epsilon_x^2 I, 4\epsilon_u^2 I)- \begin{bmatrix} X_m^- \\ U_m^- \end{bmatrix}\begin{bmatrix} X_m^- \\ U_m^- \end{bmatrix}^\top.
    \end{aligned}
\end{equation*}
The uncertainty set $\Gamma$ is non-empty as it must contain $A_*,B_*$. Since \cref{data richness assumption} holds and $T\cdot\max(4\epsilon_x^2,4\epsilon_u^2)\leq \gamma$, we have $\mathbf{A}\prec0$. Therefore, we can conclude that $\Gamma$ is bounded following a similar argument as in \cite{bisoffi2022data}. The relationship $\Sigma \subseteq \Gamma$ can be proved using a similar argument as in \cite{bisoffi2021trade}. $\hfill \square$
}

\subsection{Proof of \cref{stability without noise theorem}}\label{proof of stability without noise theorem}
Given that Problem \eqref{stability without noise theorem eq} is feasible, we first apply the change of variables $P = \Bar{P}^{-1}, \alpha_1 = \Bar{\alpha}_1^{-1}, \alpha_2 = \Bar{\alpha}_2^{-1}$ to \eqref{stability without noise theorem eq1}, which results in 
\begin{equation}\label{temp1 in proof of thr1}
    \begin{aligned}
        &\scalebox{0.98}{$\begin{bmatrix}
            \mathcal{T}(P^{-1}) &\mathcal{B}_1(P^{-1},L,-G) &0\\
            * &\Delta &\mathcal{B}_2(L^\top,P^{-1})\\
            * &* &\mathrm{blkdiag}(\frac{1}{\alpha_2\sigma_u^2}I, \frac{1}{\alpha_1\sigma_x^2}I)
        \end{bmatrix}\succeq0,$}\\
        & \Delta = \mathrm{blkdiag}(\beta P^{-1}, \Omega_1, \Theta_2).
    \end{aligned}
\end{equation}
According to $(\alpha_1^{-1}I-P^{-1})^\top\alpha_1 I(\alpha_1^{-1}I-P^{-1})\succeq0$ and $(\alpha_2^{-1}I-G)^\top\alpha_2 I(\alpha_2^{-1}I-G)\succeq0$, we obtain two inequalities
\begin{equation*}
\begin{aligned}
    \alpha_1 P^{-1}P^{-1}&\succeq 2P^{-1}-\alpha_1^{-1}I, \\
    \alpha_2 G^\top G&\succeq G^\top+G-\alpha_2^{-1}I.
\end{aligned}
\end{equation*}
Therefore, \eqref{temp1 in proof of thr1} imples 
\begin{equation}\label{temp2 in proof of thr1}
    \begin{aligned}
        &\scalebox{0.97}{$\begin{bmatrix}
            \mathcal{T}(P^{-1}) &\mathcal{B}_1(P^{-1},L,-G) &0\\
            * &\Delta &\mathcal{B}_2(L^\top,P^{-1})\\
            * &* &\mathrm{blkdiag}(\frac{1}{\alpha_2\sigma_u^2}I, \frac{1}{\alpha_1\sigma_x^2}I)
        \end{bmatrix}\succeq0,$}\\
        & \Delta = \mathrm{blkdiag}(\beta P^{-1}, \alpha_1(P^{-1})^2, \alpha_2 G^\top G).
    \end{aligned}
\end{equation}
By introducing a variable $L = KP^{-1}$, and pre- and post-multiplying \eqref{temp2 in proof of thr1} by $\text{blkdiag}(I,I,I,P,P,(G^\top)^{-1},I,I)$ and its transpose, we obtain
\begin{equation*}
    \begin{aligned}
        &\begin{bmatrix}
            \mathcal{T}(P^{-1}) &\mathcal{B}_1(I,K,-I) &0\\
            * &\Delta &\mathcal{B}_2(K^\top,I)\\
            * &* &\mathrm{blkdiag}(\frac{1}{\alpha_2\sigma_u^2}I, \frac{1}{\alpha_1\sigma_x^2}I)
        \end{bmatrix}\succeq0,\\
        & \Delta = \mathrm{blkdiag}(\beta P, \alpha_1 I, \alpha_2 I).
    \end{aligned}
\end{equation*}
Applying the Schur complement, this can be transformed into
\begin{equation*}
        \begin{bmatrix}
            \mathcal{T}(P^{-1}) &\mathcal{B}_1(I,K,-I) \\
            * &\mathrm{blkdiag}(\beta P, 0, 0) - \alpha_1\Phi_1 - \alpha_2\Phi_2
        \end{bmatrix}\succeq0,
\end{equation*}
where 
\begin{equation*}
    \scalebox{0.96}{$
    \begin{aligned}
\Phi_1=\begin{bmatrix}
        \sigma_x^2 I &0 &0\\
        0       &-I &0\\
        0 &0 &0
    \end{bmatrix},\Phi_2=\begin{bmatrix}
        \begin{bmatrix}
        I\\
        -I
    \end{bmatrix}\sigma^2_u K^\top K\begin{bmatrix}
        I\\
        -I
    \end{bmatrix}^\top &\begin{bmatrix}
        0\\
        0
    \end{bmatrix}\\
    \begin{bmatrix}
        0 &0
    \end{bmatrix} &-I
    \end{bmatrix}.
    \end{aligned}
    $}
\end{equation*}
By applying the Schur complement once again, we obtain
\begin{equation*}
\begin{bmatrix}
    P^{-1} &0\\
    0  &-M_1 N_1^{-1}M_1^\top
\end{bmatrix}-\sum_{i=0}^{T-1}\theta_i\Psi_i\succeq0,
\end{equation*}
where \begin{equation*}
    M_1=\begin{bmatrix}
        I &0 &0\\
        K &-K &-I
    \end{bmatrix},
    N_1 = \mathcal{B}_4(\beta P)-\alpha_1\Phi_1-\alpha_2\Phi_2.
\end{equation*}
Thus, according to \cref{s procedure}, 
\begin{equation*}
    \begin{bmatrix}
        I\\
        A^\top\\
        B^\top
    \end{bmatrix}^\top\begin{bmatrix}
    P^{-1} &0\\
    0  &-M_1N_1^{-1}M_1^\top
\end{bmatrix}\begin{bmatrix}
        I\\
        A^\top\\
        B^\top
    \end{bmatrix}\succeq0
\end{equation*}
holds for $\forall(A,B)\in\Sigma$. This is equivalent to
\begin{equation*}
\begin{bmatrix}
    \mathcal{B}_4(\beta P)-\alpha_1\Phi_1-\alpha_2\Phi_2 &M_1^\top\begin{bmatrix}
        A^\top\\
        B^\top
    \end{bmatrix}\\
    \begin{bmatrix}
        A &B
    \end{bmatrix}M_1 &P^{-1}
\end{bmatrix}\succeq0
\end{equation*}
for $\forall(A,B)\in\Sigma$ by expansion and application of the Schur complement. Applying another Schur complement leads to
\begin{equation*}
    \mathcal{B}_4(\beta P)-M_1^\top\begin{bmatrix}
        A^\top\\
        B^\top
    \end{bmatrix}P\begin{bmatrix}
        A &B
    \end{bmatrix}M_1-\alpha_1\Phi_1-\alpha_2\Phi_2\succeq0.
\end{equation*}
This, according to \cref{s procedure}, further implies 
\begin{equation*}
\scalebox{0.94}{$
\begin{bmatrix}
    x(k)\\
    e_x(k)\\
    e_u(k)
\end{bmatrix}^\top
    \left(\mathcal{B}_4(\beta P)-M_1^\top\begin{bmatrix}
        A^\top\\
        B^\top
    \end{bmatrix}P\begin{bmatrix}
        A^\top\\
        B^\top
    \end{bmatrix}^\top M_1 \right)\begin{bmatrix}
    x(k)\\
    e_x(k)\\
    e_u(k)
\end{bmatrix}
    \succeq0
    $}
\end{equation*}
for $\forall(A,B)\in\Sigma$ and $\forall \left[\begin{smallmatrix}
    x^\top &e^\top_x(k) &e^\top_u(k)
\end{smallmatrix}\right]^\top$ satisfying \eqref{event-trigger QMI 1} and \eqref{event-trigger QMI 2}. After replacing $M_1$ and expanding, we can easily obtain 
\begin{equation*}
\begin{split}
    &\left(Ax(k)+BKx(k)-BKe_x(k)-Be_u(k)\right)^\top P\\
    &\left(Ax(k)+BKx(k)-BKe_x(k)-Be_u(k)\right) \le \beta x^\top(k) Px(k)
\end{split}
\end{equation*}
for $\forall(A,B)\in\Sigma$ and $\forall \left[\begin{smallmatrix}
    x^\top &e^\top_x(k) &e^\top_u(k)
\end{smallmatrix}\right]^\top$ satisfying \eqref{event-trigger QMI 1} and \eqref{event-trigger QMI 2}, which implies that closed-loop system \eqref{closed-loop system problem 1} is  asymptotically stable. The proof is complete.$\hfill \square$
\subsection{Proof of \cref{relationship between three problems}}\label{proof of relationship between three problems}
The first claim is trivial by leveraging the inequality \eqref{linearization of P inverse}. The second claim is also straightforward since we can always take the solution $(\Tilde{P}^{-1},\Tilde{P},K,G,\alpha_1,\Bar{\alpha}_2,\{\theta_i\}_{i=0,1,\ldots,T-1})$ to ensure the feasibility of Problem \eqref{real optimization problem of problem 1}. We now focus on the third claim. Since \eqref{temp3 eq1} and \eqref{temp3 eq2} are implied by \eqref{stability without noise theorem eq}, given a feasible solution $(\Bar{P},L,G,\Bar{\alpha}_1,\Bar{\alpha}_2,\{\theta_i\}_{i=0,1,\ldots,T-1})$ to Problem~\eqref{stability without noise theorem eq}, $(\Bar{P}^{-1},L\Bar{P}^{-1},G,\Bar{\alpha}_1^{-1},\Bar{\alpha}_2,\{\theta_i\}_{i=0,1,\ldots,T-1})$ is a feasible solution to Problem \eqref{temp3}. The third claim is thus proved by utilizing the second claim. The proof is complete.$\hfill \square$
\subsection{Proof of \cref{stability with noise theorem}}\label{proof of stability with noise theorem}
Similar to the proof of \cref{stability without noise theorem}, given that Problem \eqref{stability with noise theorem eq} is feasible, \eqref{stability with noise theorem eq1} implies
\begin{equation*}
\begin{aligned}
        &\begin{bmatrix}
            \mathcal{T}(P^{-1}-\alpha_2^{-1}EE^\top) &\mathcal{B}_1(I,K,-I) &0\\
            * & \Delta_1 &\mathcal{B}_2(K^\top,I)\\
            * &* &\Delta_2
        \end{bmatrix}\succeq0,\\
        & \Delta_1 = \mathrm{blkdiag}((\beta-\alpha_1) P, \alpha_3 I, \alpha_4 I),\\
        & \Delta_2 = \mathrm{blkdiag}(\frac{1}{\alpha_4\sigma_u^2}I, \frac{1}{\alpha_3\sigma_x^2}I ),
\end{aligned}
\end{equation*}
by the change of variables $\Bar{P}=P^{-1},L=KP^{-1},\alpha_2 = \Bar{\alpha}_2^{-1},\alpha_3=\Bar{\alpha}_3^{-1},\alpha_4=\Bar{\alpha}_4^{-1}$, and using the inequalities
\begin{equation*}
\begin{aligned}
    \alpha_3 P^{-1}P^{-1}&\succeq 2P^{-1}-\alpha_3^{-1}I, \\
    \alpha_4 G^\top G&\succeq G^\top+G-\alpha_4^{-1}I.
\end{aligned}
\end{equation*}
Using the Schur complement, this can be transformed into

\begin{equation*}
\begin{aligned}
        &\begin{bmatrix}
            \mathcal{T}(P^{-1}-\alpha_2^{-1}EE^\top) &\mathcal{B}_1(I,K,-I)\\
            * & \Delta 
        \end{bmatrix}\succeq0,\\
        & \Delta = \mathrm{blkdiag}((\beta-\alpha_1) P, 0, 0) - \alpha_3\Phi_1 - \alpha_4\Phi_2.
\end{aligned}
\end{equation*}
By applying the Schur complement once again, we obtain
\begin{equation*}\begin{bmatrix}
    P^{-1}-\alpha_2^{-1}EE^\top &0\\
    0  &-M_1 N_2^{-1} M_1^\top
\end{bmatrix}
-\sum_{i=0}^{T-1}\theta_i\Psi_i\succeq0,
\end{equation*}
where \begin{equation*}
    N_2=\mathcal{B}_4((\beta -\alpha_1)P)-\alpha_3\Phi_1-\alpha_4\Phi_2.
\end{equation*}
Thus, according to \cref{s procedure}, 
\begin{equation*}
    \begin{bmatrix}
        I\\
        A^\top\\
        B^\top
    \end{bmatrix}^\top\begin{bmatrix}
    P^{-1}-\alpha_2^{-1}EE^\top &0\\
    0  &N_2
\end{bmatrix}\begin{bmatrix}
        I\\
        A^\top\\
        B^\top
    \end{bmatrix}\succeq0
\end{equation*}
holds for $\forall(A,B)\in\Sigma$. After expanding the left-hand side, we leverage $\alpha_1-\alpha_2\epsilon_d^2\ge0$ and apply the Schur complement, which leads to
\begin{equation*}
\begin{bmatrix}
    \alpha_1-\alpha_2\epsilon_d^2 &0 &0 \\
    0 &N_2 &M_1^\top\begin{bmatrix}
        A^\top\\
        B^\top
    \end{bmatrix}\\
    0 &\begin{bmatrix}
        A &B
    \end{bmatrix}M_1 &P^{-1}-\alpha_2^{-1}EE^\top
\end{bmatrix}\succeq0.
\end{equation*}
By applying the Schur complement, the matrix inversion lemma $(P^{-1}-\alpha_2^{-1}EE^\top)^{-1}=P-PE(E^\top PE-\alpha_2 I)^{-1}EP$, and another Schur complement, we can obtain
\begin{equation*}
\begin{bmatrix}
    \alpha_1-\alpha_2\epsilon_d^2 &0 &0\\
0 &N_2-R_{11} &-R_{12}\\
 0 &-R_{21}&\alpha_2 I-R_{22}
\end{bmatrix}\succeq0,
\end{equation*}
where 
\begin{equation*}
    \begin{split}
    &R_{11} = M_1^\top\begin{bmatrix}
                A^\top\\
                B^\top
            \end{bmatrix}P\begin{bmatrix}
                A^\top\\
                B^\top
            \end{bmatrix}^\top M_1, R_{12} = M_1^\top\begin{bmatrix}
        A^\top\\
        B^\top
    \end{bmatrix}PE,\\
    &R_{21} = R_{12}^\top, R_{22} = E^\top PE,\\
        \end{split}
\end{equation*}
which can further be reorganized as 
\begin{equation}\label{split of long equation in proof of thr2}
        \begin{bmatrix}
            0 &0 &0\\
            0 &S-R_{11} &-R_{12}\\
            0 &-R_{21}&-R_{22}
        \end{bmatrix}-\alpha_3
            \Gamma_3
            -\alpha_4
            \Gamma_4-\alpha_1\Gamma_1-\alpha_2\Gamma_2\succeq 0,
\end{equation}
where $ S = \mathcal{B}_4(\beta P), \Gamma_1 = \mathrm{blkdiag}(-1,P,0,0,0), \Gamma_2 = \mathrm{blkdiag}(\epsilon_d^2,0,0,0,-I), \Gamma_3 = \mathrm{blkdiag}(0,\Phi_1,0)$, and $\Gamma_4 = \mathrm{blkdiag}(0,\Phi_2,0)$. We will show that the above inequality can validate (i) and (ii). We first focus on (i). Consider an augmented state $\left[\begin{smallmatrix}
    1 &x^\top &e^\top_x(k) &e^\top_u(k) &d^\top(k)
\end{smallmatrix}\right]^\top$. The condition $V(k)\ge 1$ can be written as 
\begin{equation}\label{proof of thr2 con1}
    \left[\begin{smallmatrix}
    1\\
    x\\
    e_x(k)\\
    e_u(k)\\
    d(k)
\end{smallmatrix}\right]^\top
       \Gamma_1\left[\begin{smallmatrix}
    1\\
    x\\
    e_x(k)\\
    e_u(k)\\
    d(k)
\end{smallmatrix}\right]\ge0.
\end{equation}
Similarly, $||d(k)||_2\leq \epsilon_d$ can be written as
\begin{equation}\label{proof of thr2 con2}
    \left[\begin{smallmatrix}
    1\\
    x\\
    e_x(k)\\
    e_u(k)\\
    d(k)
\end{smallmatrix}\right]^\top
       \Gamma_2\left[\begin{smallmatrix}
    1\\
    x\\
    e_x(k)\\
    e_u(k)\\
    d(k)
\end{smallmatrix}\right]\ge0.
\end{equation}
Besides, using \eqref{event-trigger QMI 1} and \eqref{event-trigger QMI 2}, the event-triggered conditions $||e_x(k)||_2\le\sigma_x||x(k)||_2$ and $||e_u(k)||_2\le\sigma_u||u(k)||_2$ can also be written as
\begin{equation}\label{proof of thr2 con3}
\left[\begin{smallmatrix}
    1\\
    x\\
    e_x(k)\\
    e_u(k)\\
    d(k)
\end{smallmatrix}\right]^\top
       \Gamma_3\left[\begin{smallmatrix}
    1\\
    x\\
    e_x(k)\\
    e_u(k)\\
    d(k)
\end{smallmatrix}\right]\ge0
\end{equation}
and
\begin{equation}\label{proof of thr2 con4}
\left[\begin{smallmatrix}
    1\\
    x\\
    e_x(k)\\
    e_u(k)\\
    d(k)
\end{smallmatrix}\right]^\top
       \Gamma_4\left[\begin{smallmatrix}
    1\\
    x\\
    e_x(k)\\
    e_u(k)\\
    d(k)
\end{smallmatrix}\right]\ge0,
\end{equation}
 respectively. According to \cref{s procedure}, \eqref{split of long equation in proof of thr2} thereby implies 
\begin{equation*}
    \left[\begin{smallmatrix}
    1\\
    x\\
    e_x(k)\\
    e_u(k)\\
    d(k)
\end{smallmatrix}\right]^\top\begin{bmatrix}
            0 &0 &0\\
            0 &S-R_{11} &-R_{12}\\
            0 &-R_{21} &-R_{22}
        \end{bmatrix}\left[\begin{smallmatrix}
    1\\
    x\\
    e_x(k)\\
    e_u(k)\\
    d(k)
\end{smallmatrix}\right]\ge 0
\end{equation*}
for $\forall(A,B)\in\Sigma$ and $\forall \left[\begin{smallmatrix}
    x^\top &e^\top_x(k) &e^\top_u(k) &d^\top(k)
\end{smallmatrix}\right]^\top$ satisfying \eqref{proof of thr2 con1}, \eqref{proof of thr2 con2}, \eqref{proof of thr2 con3}, and \eqref{proof of thr2 con4}. After expansion, we obtain
\begin{equation*}
    \scalebox{0.96}{$
\begin{aligned}
    &(Ax(k)+BKx(k)-BKe_x(k)-Be_u(k)+Ed(k))^\top P(Ax(k)\\
    &+BKx(k)-BKe_x(k)-Be_u(k)+Ed(k)) \le \beta x^\top(k) Px(k).
\end{aligned}
$}
\end{equation*}
Thus, the condition (i) is validated.

We will show that \eqref{split of long equation in proof of thr2} can also validate the condition (ii). Since $\beta\ge\alpha_1\ge0$, we can change a variable $\Bar{\alpha}_1=\beta-\alpha_1\ge0$. By further leveraging $\beta<1$, \eqref{split of long equation in proof of thr2} implies
\begin{equation}\label{split of long equation in proof of thr2 part 2}
    \begin{split}
        \begin{bmatrix}
            1 &0 &0\\
            0 &-R_{11} &-R_{12}\\
            0 &-R_{21} &-R_{22}
        \end{bmatrix}-\alpha_3
            \Gamma_3
            -\alpha_4
            \Gamma_4-\Bar{\alpha}_1\Bar{\Gamma}_1-\alpha_2\Gamma_2\succeq 0,
    \end{split}
\end{equation}
where $\Bar{\Gamma}_1 = \mathrm{blkdiag}(1,-P,0,0,0)$. The condition $V(k)\le 1$ can be written as
\begin{equation}\label{proof of thr2 con5}
    \left[\begin{smallmatrix}
    1\\
    x\\
    e_x(k)\\
    e_u(k)\\
    d(k)
\end{smallmatrix}\right]^\top
       \Bar{\Gamma}_1\left[\begin{smallmatrix}
    1\\
    x\\
    e_x(k)\\
    e_u(k)\\
    d(k)
\end{smallmatrix}\right]\ge0.
\end{equation}
According to \cref{s procedure}, \eqref{split of long equation in proof of thr2 part 2} thereby implies
\begin{equation*}
    \left[\begin{smallmatrix}
    1\\
    x\\
    e_x(k)\\
    e_u(k)\\
    d(k)
\end{smallmatrix}\right]^\top \begin{bmatrix}
            1 &0 &0\\
            0 &-R_{11} &-R_{12}\\
            0 &-R_{21} &-R_{22}
        \end{bmatrix}\left[\begin{smallmatrix}
    1\\
    x\\
    e_x(k)\\
    e_u(k)\\
    d(k)
\end{smallmatrix}\right]\ge0
\end{equation*}
for $\forall(A,B)\in\Sigma$ and $\forall \left[\begin{smallmatrix}
    x^\top &e^\top_x(k) &e^\top_u(k) &d^\top(k)
\end{smallmatrix}\right]^\top$ satisfying \eqref{proof of thr2 con5}, \eqref{proof of thr2 con2}, \eqref{proof of thr2 con3}, and \eqref{proof of thr2 con4}. After expansion, we obtain
\begin{equation*}
\begin{split}
    &\left(Ax(k)+BKx(k)-BKe_x(k)-Be_u(k)+Ed(k)\right)^\top P\\
    &\left(Ax(k)+BKx(k)-BKe_x(k)-Be_u(k)+Ed(k)\right) \le 1,
\end{split}
\end{equation*}
which validates the condition (ii). The proof is complete. $\hfill \square$

\subsection{Proof of \cref{monotonicity pro}}\label{proof of monotonicity pro}
Consider two positive scalars $c_1,c_2$ with $c_1> c_2$, then there exists a scalar $n>1$ such that $c_1=n c_2$. Assume that Problem \eqref{stability with noise theorem eq} is feasible when $\alpha_1=c_1$ with the solution $
(\Bar{P}, L, G,\Bar{\alpha}_2,\Bar{\alpha}_3, \Bar{\alpha}_4, \{\theta_i\}_{i=0,1,\ldots,T-1})$. By scaling, Problem \eqref{stability with noise theorem eq} is feasible when $\alpha_1=c_2$ with the solution $
(n\Bar{P}, nL, nG,n\Bar{\alpha}_2,n\Bar{\alpha}_3, n\Bar{\alpha}_4, \{n\theta_i\}_{i=0,1,\ldots,T-1})$ because $(\beta-c_2)n\Bar{P}=(\beta-c_1/n)n\Bar{P}\succeq (\beta-c_1) n\Bar{P}$.$\hfill \square$
{\color{blue}
\subsection{Experiment Settings and Implementation Procedures}
In this subsection, we present the experimental settings and implementation procedures to enhance reproducibility. For clarity, let $T_p$ denote the simulation length. The common parameter settings for \textit{Examples 1--3} are summarized below.
\begin{table}[H]
\centering{\color{blue}
\begin{tabular}{ccccccc}
\toprule
\textbf{Parameter}&Random Seed  & $T_s$ & $T$ & $\epsilon_x$ & $\epsilon_u$ & $\epsilon_r$ \\
\midrule
\textbf{Value}&rng(123) & $0.1$ & $100$ & $0.001$ &$0.001$ &$10^{-6}$ \\
\bottomrule
\end{tabular}}
\end{table}

The parameter settings and implementation procedure for \textit{Example 1} are summarized below.
\begin{table}[H]
\centering{\color{blue}
\begin{tabular}{ccccccc}
\toprule
\textbf{Parameter}& $\epsilon_d$  & $\sigma_x$ &$\sigma_u$   &$T_p$ &$\beta$\\
\midrule
\textbf{Value}&$0$  &$0.1581$ &$0.01$  &$60$ &$0.999$\\
\bottomrule
\end{tabular}}
\end{table}

\begin{table}[H]
\centering{\color{blue}
\begin{tabular}{c p{0.8\columnwidth}}
\toprule
\textbf{Step} & \textbf{Description} \\
\midrule
1 & Discretize the REA1 model using the ZOH method with sampling period $T_s$ to obtain $A_*$ and $B_*$.\\
2 & Starting from $\left[\begin{smallmatrix} 0 & 0 & 0 & 0 \end{smallmatrix}\right]^\top$, apply a sequence of inputs $\{u(k)\}_{k=0}^{T-1}$ to the system with $u_i(k)$ uniformly sampled from $[-1,1]$, and collect $X_m$ and $U_m$ under $\|\delta_x(k)\|_2 \leq \epsilon_x$ and $\|\delta_u(k)\|_2 \leq \epsilon_u$. \\
3 & Run \cref{algorithm 1} to obtain the controller gain $K$. \\
4 & Starting from $\left[\begin{smallmatrix} -1 & -2 & 2 & -2 \end{smallmatrix}\right]^\top$, simulate the system under the feedback control law with $K$ and the event-triggered mechanism.\\
5 & Compute the transmission rates $r_x$ and $r_u$. \\
\bottomrule
\end{tabular}}
\end{table}

The parameter settings and implementation procedure for \textit{Example 2} are summarized below.
\begin{table}[H]
\centering{\color{blue}
\begin{tabular}{cccccccc}
\toprule
\textbf{Parameter}& $\epsilon_d$  &$\alpha_1$ & $\sigma_x$ &$\sigma_u$   & $T_p$ &$\beta$\\
\midrule
\textbf{Value}&$0.1$  &$0.05$ &$0.1414$ &$0.3162$   &$60$ &$0.999$\\
\bottomrule
\end{tabular}}
\end{table}

\begin{table}[H]
\centering{\color{blue}
\begin{tabular}{c p{0.8\columnwidth}}
\toprule
\textbf{Step} & \textbf{Description} \\
\midrule
1 & Discretize the mass-spring system using the ZOH method with sampling period $T_s$ to obtain $A_*$, $B_*$ and $E$.\\
2 & Starting from $\left[\begin{smallmatrix} 0 & 0 & 0 & 0 \end{smallmatrix}\right]^\top$, apply a sequence of inputs $\{u(k)\}_{k=0}^{T-1}$ to the system with $u_i(k)$ uniformly sampled from $[-1,1]$, and collect $X_m$ and $U_m$ under $\|\delta_x(k)\|_2 \leq \epsilon_x$, $\|\delta_u(k)\|_2 \leq \epsilon_u$, and  $\|d(k)\|_2 \leq \epsilon_d$. \\
3 & Run \cref{algorithm 2} with different values of $\lambda$ to obtain the corresponding controller gains $K$. \\
4 & Starting from $\left[\begin{smallmatrix}1 &-2 &2 &-1\end{smallmatrix}\right]^\top$, simulate the system under the feedback control law with $K$ and the event-triggered mechanism.\\
5 & Compute the transmission rates $r_x$ and $r_u$. \\
\bottomrule
\end{tabular}}
\end{table}

The parameter settings and implementation procedure for \textit{Example 3} are summarized below.
\begin{table}[H]
\centering{\color{blue}
\begin{tabular}{ccccccc}
\toprule
\textbf{Parameter}&  $\epsilon_d$  & $\sigma_x$ &$\sigma_u$  & $T_p$\\
\midrule
\textbf{Value}& $0.1$  &$0.1414$ &$0.4472$  &$100$\\
\bottomrule
\end{tabular}}
\end{table}

\begin{table}[H]
\centering{\color{blue}
\begin{tabular}{c p{0.8\columnwidth}}
\toprule
\textbf{Step} & \textbf{Description} \\
\midrule
1 & Discretize the HE1 model using the ZOH method with sampling period $T_s$ to obtain $A_*$, $B_*$, $C$, $D$, and $E$.\\
2 & Starting from $\left[\begin{smallmatrix} 0 & 0 & 0 & 0 \end{smallmatrix}\right]^\top$, apply a sequence of inputs $\{u(k)\}_{k=0}^{T-1}$ to the system with $u_i(k)$ uniformly sampled from $[-1,1]$, and collect $X_m$ and $U_m$ under $\|\delta_x(k)\|_2 \leq \epsilon_x$, $\|\delta_u(k)\|_2 \leq \epsilon_u$, and  $\|d(k)\|_2 \leq \epsilon_d$. \\
3 & Run \cref{algorithm 3} with different values of $\lambda$ to obtain the corresponding controller gains $K$. \\
4 & Starting from $\left[\begin{smallmatrix}0 &0 &0 &0\end{smallmatrix}\right]^\top$, simulate the system under the feedback control law with $K$ and the event-triggered mechanism.\\
5 & Compute the transmission rates $r_x$ and $r_u$. \\
\bottomrule
\end{tabular}}
\end{table}
}
\footnotesize{ 
\bibliographystyle{IEEEtran}
\bibliography{reference}
}
	\begin{IEEEbiography}[{\includegraphics[width=1in,height=1.25in,clip,keepaspectratio]{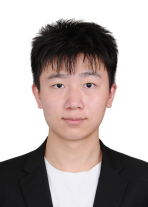}}]{Zhaohua Yang} received the B.Eng. degree from Zhejiang University, Hangzhou, China, and the B.S. degree from University of Illinois at Urbana-Champaign, Urbana, USA, in 2022. He is currently pursuing the Ph.D. degree with the Department of Electronic and Computer Engineering at The Hong Kong University of Science and Technology, Hong Kong, China. From August 2021 to January 2022, he was a visiting student at the Grainger College of Engineering, University of Illinois at Urbana-Champaign, Urbana, USA.
		His research interests include optimal control, data-driven control, and sparse optimization.
	\end{IEEEbiography}

	\begin{IEEEbiography}[{\includegraphics[width=1in,height=1.25in,clip,keepaspectratio]{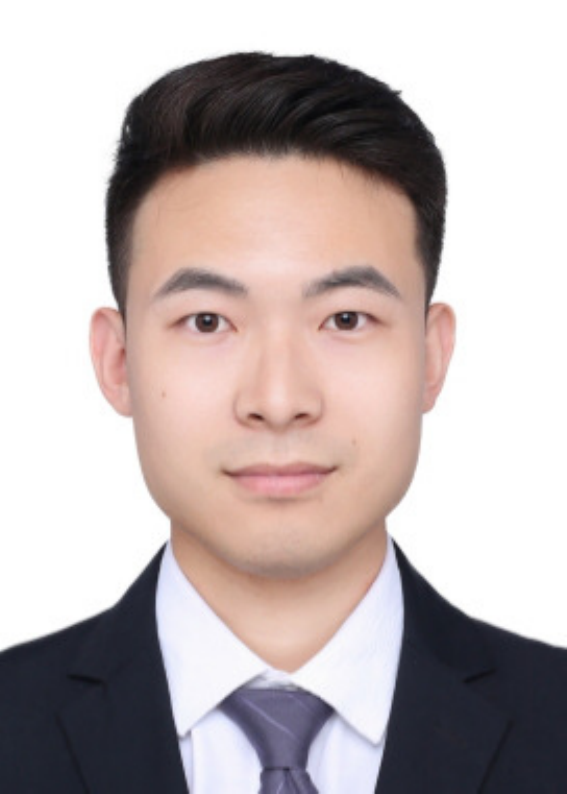}}]{Xiaoxu Lyu} received his B.Eng. degree in Naval Architecture and Marine Engineering from Harbin Institute of Technology in 2018 and his Ph.D. degree in Dynamical Systems and Control from Peking University in 2023. He is currently a Research Assistant Professor in the Department of Electronic and Computer Engineering at The Hong Kong University of Science and Technology (HKUST). From 2024 to 2025, he was a Postdoctoral Fellow at HKUST. His research interests include distributed estimation and control, data-driven systems, and multi-robot systems.
	\end{IEEEbiography}

	\begin{IEEEbiography}[{\includegraphics[width=1in,height=1.25in,clip,keepaspectratio]{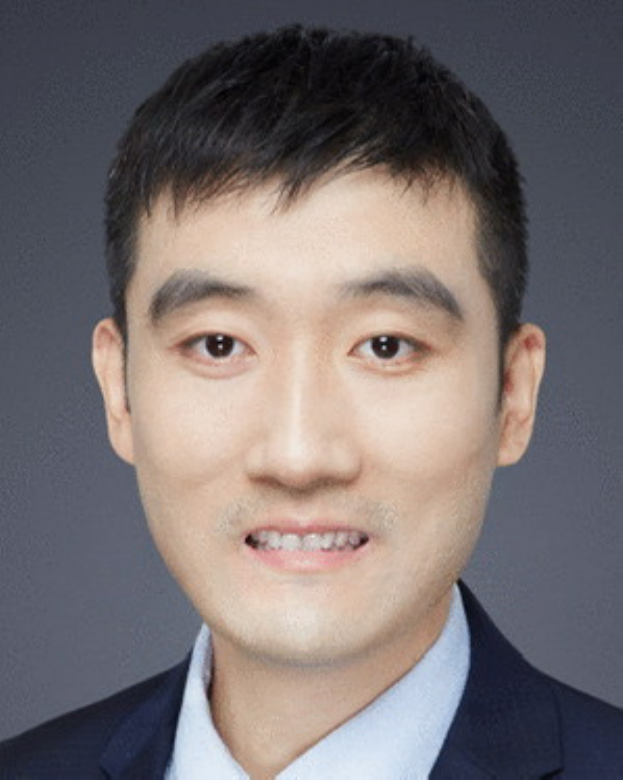}}]{Dawei Shi} received the B.E. degree in electrical
engineering and automation from Beijing Insti-
tute of Technology, Beijing, China, in 2008, and
the Ph.D. degree in control systems from the
University of Alberta, Edmonton, AB, Canada,
in 2014. In December 2014, he was appointed
as an Associate Professor at the School of Au-
tomation, Beijing Institute of Technology. From
February 2017 to July 2018, he was with the
Harvard John A. Paulson School of Engineering
and Applied Sciences, Harvard University, as a
Post-Doctoral Fellow of bioengineering. Since July 2018, he has been
with the School of Automation, Beijing Institute of Technology, where he
is currently a Professor. His research interests include the analysis and
synthesis of complex sampled-data control systems with applications to
biomedical engineering, robotics, and motion systems. He is a member
of the Early Career Advisory Board of Control Engineering Practice. He
serves as an Associate Editor/a Technical Editor for IEEE Transactions
on Industrial Electronics, IEEE/ASME Transactions on Mechatronics,
IEEE Control Systems Letters, and IET Control Theory and Applications.
He was a Guest Editor of European Journal of Control. He served as an
Associate Editor for IFAC World Congress and a member for the IEEE
Control Systems Society Conference Editorial Board.
	\end{IEEEbiography}
    
		\begin{IEEEbiography}[{\includegraphics[width=1in,height=1.25in,clip,keepaspectratio]{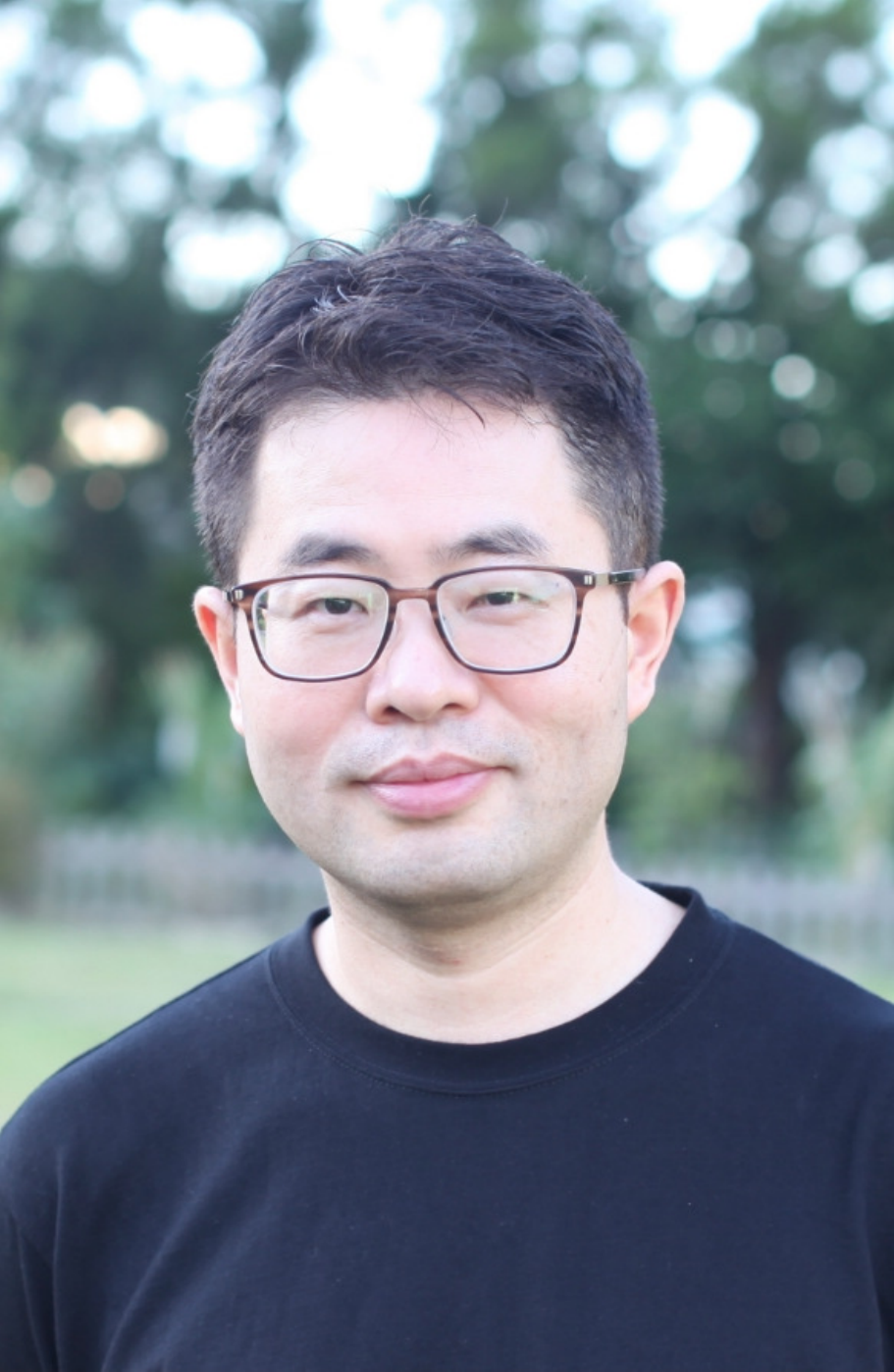}}]{Ling Shi} received his B.E. degree in Electrical and Electronic Engineering from The Hong Kong University of Science and Technology (HKUST) in 2002 and  Ph.D. degree in Control and Dynamical Systems from The California Institute of Technology (Caltech) in 2008. He is currently a Professor in the Department of Electronic and Computer Engineering at HKUST with a joint appointment in the Department of Chemical and Biological Engineering (2025-2028), and the Director of The Cheng Kar-Shun Robotics Institute (CKSRI). His research interests include cyber-physical systems security, networked control systems, sensor scheduling, event-based state estimation, and multi-agent robotic systems (UAVs and UGVs). He served as an editorial board member for the European Control Conference 2013-2016. He was a subject editor for International Journal of Robust and Nonlinear Control (2015-2017), an associate editor for IEEE Transactions on Control of Network Systems (2016-2020), an associate editor for IEEE Control Systems Letters (2017-2020), and an associate editor for a special issue on Secure Control of Cyber Physical Systems in IEEE Transactions on Control of Network Systems (2015-2017). He also served as the General Chair of the 23rd International Symposium on Mathematical Theory of Networks and Systems (MTNS 2018). He is currently serving as a member of the Engineering Panel (Joint Research Schemes) of the Hong Kong Research Grants Council (RGC) (2023-2026). He received the 2024 Chen Han-Fu Award given by the Technical Committee on Control Theory, Chinese Association of Automation (TCCT, CAA). He is a member of the Young Scientists Class 2020 of the World Economic Forum (WEF), a member of The Hong Kong Young Academy of Sciences (YASHK), and he is an IEEE Fellow.
	\end{IEEEbiography}
\end{document}